\documentclass[11pt,reqno]{amsart}

\usepackage{amsmath, amsthm, amssymb, esint}
\usepackage{amsfonts}
\usepackage{bbm}
\usepackage[ansinew]{inputenc}
\usepackage[dvips]{epsfig}
\usepackage{graphicx}
\usepackage[english]{babel}
\usepackage{enumerate}
\usepackage{enumitem}
\usepackage{hyperref}
\usepackage{mathrsfs}
\usepackage{color}

\newlist{steps}{enumerate}{1}
\setlist[steps, 1]{label = {\it \underline{\smash{Step~\arabic*:}}}, ref={\it Step~\arabic*}, wide, itemindent = 3pt}
 
\theoremstyle{plain}
\newtheorem{thm}{Theorem}[section]

\newtheorem{lem}[thm]{Lemma}
\newtheorem{prop}[thm]{Proposition}

\theoremstyle{definition}
\newtheorem{defi}[thm]{Definition}

\theoremstyle{remark}
\newtheorem{rem}[thm]{Remark}

\numberwithin{equation}{section}

\newcommand{\X}{\mathfrak{X}}

\newcommand{\R}{\mathbb{R}}

\newcommand{\Z}{\mathbb{Z}}
\newcommand{\N}{\mathbb{N}}

\newcommand{\F}{\mathcal{F}}
\newcommand{\eps}{\varepsilon}

\def\MR{\mathcal{R}}

\def\R{\mathbb{R}}

\def\Z{\mathbb{Z}}
\def\N{\mathbb{N}}

\def\L{\mathcal{L}}

\def\A{\mathcal{A}}

\newcommand{\G}{{\mathfrak{G}}}
\newcommand{\GL}{{\mathfrak{G}_s(\lambda, \Lambda)}}

\DeclareRobustCommand{\S}{\ifmmode\mathsection\else\textsection\fi}
\renewcommand\mathsection{\operatorname{\mathbb{S}}}

\DeclareMathOperator{\dist}{dist}

\newcommand{\average}{{\mathchoice {\kern1ex\vcenter{\hrule height.4pt
width 6pt depth0pt} \kern-9.7pt} {\kern1ex\vcenter{\hrule
height.4pt width 4.3pt depth0pt} \kern-7pt} {} {} }}

\def\R{\mathbb{R}}

\title[Estimates for nonlocal equations with singular kernels]{Schauder and Cordes-Nirenberg estimates \\ for nonlocal elliptic equations \\ with singular kernels}

\author{Xavier Fern\'andez-Real}
\address{EPFL SB, Station 8, CH-1015 Lausanne, Switzerland}
\email{xavier.fernandez-real@epfl.ch}

\author{Xavier Ros-Oton}
\address{ICREA, Pg. Llu\'is Companys 23, 08010 Barcelona, Spain \& Universitat de Barcelona, Departament de Matem\`atiques i Inform\`atica, Gran Via de les Corts Catalanes 585, 08007 Barcelona, Spain \& Centre de Recerca Matem\`atica, Barcelona, Spain}
\email{xros@icrea.cat}

\keywords{Integro-differential equations, Schauder estimates, Cordes-Nirenberg estimates.}

\subjclass[2010]{35B65, 35R11, 47G20}

\thanks{X. F. was supported by the Swiss National Science Foundation (SNF grants 200021\_182565 and PZ00P2\_208930),   by the Swiss State Secretariat for Education, Research and Innovation (SERI) under contract number MB22.00034, and by the AEI project PID2021-125021NA-I00 (Spain). \\X. R. was supported by the European Research Council (ERC) under the Grant Agreement No 801867, by the AEI project PID2021-125021NA-I00 (Spain), by the AGAUR project 2021 SGR 00087 (Catalunya), the AEI grant RED2022-134784-T funded by MCIN/AEI/10.13039/501100011033 (Spain), and the AEI Mar\'ia de Maeztu Program for Centers and Units of Excellence in R\&D (CEX2020-001084-M)}

\begin{document}

\begin{abstract}
We study integro-differential elliptic equations (of order $2s$) with variable coefficients, and prove the natural and most general Schauder-type estimates that can hold in this setting, both in divergence and non-divergence form.
Furthermore, we also establish H\"older estimates for general elliptic equations with no regularity assumption on $x$, including for the first time operators like $\sum_{i=1}^n(-\partial^2_{\textbf{v}_i(x)})^s$, provided that the coefficients have ``small oscillation''.
\end{abstract}

\maketitle

\section{Introduction}

The study of integro-differential elliptic equations has been an important area of research in the 21st century, especially since the celebrated works of Caffarelli and Silvestre \cite{CS09,CS11b,CS11}.
There are several important motivations to study these equations, from quite diverse areas including Probability (L\'evy processes) \cite{Sat99}, Mathematical Physics (the Boltzmann equation) \cite{IS,IS3,SiICM}, Geometry (nonlocal minimal surfaces) \cite{CFS,CabreCozzi}, or other models in applied sciences \cite{DGLZ,Eri02,GO08,MK}.
We refer to the books \cite{PK,FR23} for an overview of the state of the art in this area of research.

One of the most fundamental questions in this context is to understand the regularity of solutions to \emph{linear} equations (of order $2s$) of the form
\begin{equation}\label{linear}
\L(u,x)=f\quad \textrm{in}\quad B_1\subset \R^n.
\end{equation}
When the operator $\L$ is translation invariant, the regularity theory for \eqref{linear} is well understood \cite{FR23}.
However, there are several open problems in case of equations with {variable coefficients}, of which we have two types\footnote{The kernel $K$ may be a nonnegative measure.
Even though we use the notation $K(x,y)dy$ as if it was absolutely continuous, it does not need to be.
We will prove all our results for general measures $K$, but in the introduction we abuse notation for simplicity.}:

\vspace{2mm}

\begin{itemize}[leftmargin=*]
\item \emph{Non-divergence-form} operators are those of the form 
\begin{equation}
\label{eq:wedenotedA}
\L(u,x)  = {\rm P.V.} \int_{\R^n}\big(u(x)-u(x+y)\big)K(x,y)\, dy 
\end{equation}
with 
\begin{equation}
\label{eq:wedenotedA2}
K\geq0\qquad \textrm{and}\qquad K(x,y)=K(x,-y)\quad \textrm{for all}\quad x,y\in \R^n.
\end{equation}
They correspond (in the limiting case $s=1$) to operators of the type ${\rm tr}(A(x)D^2 u)= \sum_{i,j=1}^n a_{ij}(x)\partial_{ij}u$.

\vspace{2mm}

\item \emph{Divergence-form} operators, instead, are those of the form
\begin{equation}
\label{eq:wedenotedB}
\L (u,x)  = {\rm P.V.} \int_{\R^n}\big(u(x)-u(z)\big)K(x,z) \, dz
\end{equation}
with 
\begin{equation}
\label{eq:wedenotedB2}
K\geq0\qquad \textrm{and}\qquad K(x,z)=K(z,x)\quad \textrm{for all}\quad x,z\in \R^n.
\end{equation}
They correspond (in the limiting case $s=1$) to operators of the type ${\rm div}(A(x)\nabla u)= \sum_{i,j=1}^n \partial_i\big(a_{ij}(x)\partial_{j}u\big)$.

\vspace{2mm}

\end{itemize}

Of course, when the kernels do not depend on $x$ (translation invariant operators), then these two classes of operators coincide.

When the kernels $K$ in \eqref{eq:wedenotedA} or \eqref{eq:wedenotedB} satisfy a strong form of uniform ellipticity,
\begin{equation}\label{strong-ellipticity}
K(x,y)\asymp \frac{1}{|y|^{n+2s}} \qquad \textrm{or} \qquad K(x,z)\asymp \frac{1}{|x-z|^{n+2s}},
\end{equation}
then the regularity theory for \eqref{linear} is well understood; see \cite{BL02,Bas09,BFV,DJZ,Fal,JX,Kas09,Kri13,KMS, Ser2}.
However, things become much more delicate when one tries to understand general uniformly elliptic operators of order $2s$ (for which the kernels $K$ may be a purely singular measure). 

The ellipticity conditions in this setting are
\begin{equation}
\label{eq:Kellipt_gen_A}
r^{2s} \int_{B_{2r}\setminus B_r} K(x,y)\, dy \le \Lambda\qquad\text{for all}\quad r>0
\end{equation}
and
\begin{equation}
\label{eq:Kellipt_gen_A2}
r^{2s-2}\inf_{e\in \S^{n-1}}\int_{B_{r}}|e\cdot y|^{2} K(x,y)\, dy \geq \lambda>0 \qquad\text{for all}\quad r>0
\end{equation}
for \eqref{eq:wedenotedA}, or the analogous conditions  \eqref{divergence-ellipticity0B}-\eqref{divergence-ellipticity1B} in case of operators \eqref{eq:wedenotedB}.

For operators with no $x$-dependence, these two conditions are equivalent to $\A(\xi) \asymp |\xi|^{2s}$, where $\A$ is the Fourier symbol of $\L$; see \cite{FR23}.

An outstanding open question in this context is the following:

\vspace{2mm}

\noindent {\bf Open question:} {\it Establish H\"older estimates for solutions to $\L(u,x)=0$ for general operators $\L$ of the form  \eqref{eq:wedenotedA}-\eqref{eq:wedenotedA2} satisfying \eqref{eq:Kellipt_gen_A}-\eqref{eq:Kellipt_gen_A2}.}

\vspace{2mm}

Moreover, the analogous question is also open in case of divergence-form operators \eqref{eq:wedenotedB}-\eqref{eq:wedenotedB2}; see \cite{DK20,CS20}.

These are often called equations with \emph{bounded measurable coefficients}.
Notice that, as said before, the general ellipticity conditions \eqref{eq:Kellipt_gen_A}-\eqref{eq:Kellipt_gen_A2} allow for kernels $K$ that may be purely singular measures.
A typical example is 
\begin{equation}\label{typical}
\L(u,x):=\sum_{i=1}^n\big(-\partial^2_{\textbf{v}_i(x)}\big)^s,
\end{equation}
where $\textbf{v}_i:\R^n\to\S^{n-1}$ are such that $\det (\textbf{v}_i)\geq c>0$.

Despite many important developments in the last years \cite{SS,IS2,Loh23,DK20,IS,CS20,Now0,Now}, no regularity result is known for operators of the type \eqref{typical} ---except for the simplest case of translation invariant ones \cite{RS-stable,DRSV22,FR23}.

The goal of this paper is to prove new regularity estimates for general elliptic operators $\L$ with $x$-dependence, both in divergence and non-divergence form, under the natural ellipticity conditions \eqref{eq:Kellipt_gen_A}-\eqref{eq:Kellipt_gen_A2} for integro-differential operators (in particular, much more general than \eqref{strong-ellipticity}). 
The estimates we prove are either Cordes-Nirenberg or Schauder for equations \eqref{linear}.

\subsection{Main results for non-divergence form equations}

Our first result is a Cordes-Nirenberg-type estimate for general nonlocal equations in non-divergence form, and reads as follows.  

\begin{thm}\label{thmA}
Let $s\in (0,1)$, $\beta\in(0,2s)$, and let $\L$ be an operator of the form \eqref{eq:wedenotedA}-\eqref{eq:wedenotedA2} satisfying \eqref{eq:Kellipt_gen_A}-\eqref{eq:Kellipt_gen_A2} uniformly in $x\in \R^n$, and
\begin{equation}\label{close-kernels}
\sup_{\begin{subarray}{c} [\phi]_{C^\beta(\R^n)}\leq 1 \\ \phi(0)=0 \end{subarray}} 
\left|\int_{B_{2\rho}\setminus B_\rho} \phi(y)\big(K(x,y)-K(x',y)\big)dy\right| \leq \eta \rho^{\beta-2s},
\end{equation}
for all $x, x'\in B_1$ and all $\rho > 0$, for some $\eta > 0$.
Then, there exists $\eta_0>0$ depending only on $n$, $s$, $\beta$, $\lambda$, and $\Lambda$, such that if $\eta \le \eta_0$, then the following holds.

Let $f\in L^\infty(B_1)$, and let $u\in C_{\rm loc}^{2s+}(B_1) \cap L^\infty(\R^n)$ be any solution of \eqref{linear}.
Then,
\[
\|u\|_{C^{\beta}(B_{1/2})} \le C\left(\|u\|_{L^\infty(\R^n)} + \|f\|_{L^\infty(B_1)}\right)
\]
for some $C$ depending only on $n$, $s$, $\beta$, $\lambda$, and $\Lambda$.
\end{thm}

To the best of our knowledge, this is the first H\"older estimate for equations in non-divergence form with no regularity assumption on $x$ which allows general kernels  \eqref{eq:Kellipt_gen_A}-\eqref{eq:Kellipt_gen_A2}.
Of course, this comes at the cost of assuming that the ``coefficients'' $K(x,\cdot)$ have ``small oscillation''.

Notice that, in \eqref{close-kernels}, the ``distance'' between the kernels $K(x,\cdot)$ and $K(x',\cdot)$ is measured in a (weighted) Kantorovich-type norm.
Notice also that this assumption allows for operators with purely singular kernels, such as \eqref{typical} provided that 
\[\big|\textbf{v}_i(x)-\textbf{v}_i(x')\big|^\beta\leq c\eta\quad\textrm{in}\quad B_1.\]

\begin{rem}
Condition \eqref{close-kernels} is scale invariant, and therefore when the kernels $K$ are homogeneous in the $y$-variable, it is equivalent to
\[
\sup_{\|\phi\|_{C^\beta(\S^{n-1})}\leq 1} 
\left|\int_{\S^{n-1}} \phi(\theta)\big(K(x,\theta)-K(x',\theta)\big)d\theta\right| \leq \eta .
\]
Moreover, something similar happens for the condition \eqref{Calpha-kernels} below.
\end{rem}

Notice also that the assumption \eqref{close-kernels} is significantly weaker than
\[\qquad \qquad \int_{B_{2\rho}\setminus B_\rho} \big|K(x,y)-K(x',y)\big|dy \leq \eta \rho^{-2s}\quad \textrm{for all}\quad x, x'\in B_1,\ \rho > 0,\]
which allows some operators with singular kernels, but not those like \eqref{typical}.

We also prove a Schauder-type estimate for this type of equations, in which $C^\alpha$ regularity of the coefficients yield $C^{2s+\alpha}$ regularity of the solution.

\begin{thm}\label{thmB}
Let $s\in (0,1)$, $\gamma\in[0,1)$, and let $\L$ be an operator of the form \eqref{eq:wedenotedA}-\eqref{eq:wedenotedA2} satisfying \eqref{eq:Kellipt_gen_A}-\eqref{eq:Kellipt_gen_A2} uniformly in $x\in \R^n$, and \footnote{When $\gamma=0$, we denote $[\phi]_{C^0(\Omega)}:={\rm osc}_\Omega\,\phi$.}
\begin{equation}\label{Calpha-kernels}
\sup_{\begin{subarray}{c} [\phi]_{C^\gamma(\R^n)}\leq 1 \\ \phi(0)=0 \end{subarray}} 
\left|\int_{B_{2\rho}\setminus B_\rho} \phi(y)\big(K(x,y) - K(x',y)\big)dy\right| \leq M|x-x'|^\alpha\rho^{\gamma-2s}
\end{equation}
for all $x,x'\in \R^n$ and all $\rho>0$, for some $M > 0$. 
 In case $\alpha>\gamma$, assume in addition that $\sup_{x\in B_1} [K_x]_{\alpha-\gamma} \leq M$; see \eqref{Calpha-assumption}.

Let $f\in C^\alpha(B_1)$, and let $u\in C^{2s+\alpha}(B_1)\cap C^{\gamma}(\R^n)$ be any solution of \eqref{linear}.
Then,
\[
\|u\|_{C^{2s+\alpha}(B_{1/2})} \le C\left(\|f\|_{C^\alpha(B_1)} + \|u\|_{C^\gamma(\R^n)}\right)
\]
for some $C$ depending only on $n$, $s$, $\alpha$, $\gamma$, $\lambda$, $\Lambda$, and $M$.
\end{thm}

As explained later on, we expect \eqref{Calpha-kernels} to be the minimal assumption under which these Schauder-type estimates hold; see \eqref{nyi} and Lemma \ref{lem-sdhhh}.

Notice that, when $\gamma\geq \alpha$, our result allows for operators with purely singular kernels, like~\eqref{typical}.
This is the first Schauder-type estimate for general non-translation invariant operators, like \eqref{typical}.
On the other hand, under condition \eqref{Calpha-kernels} with $\gamma=0$, a similar result was established in~\cite{IS2} in the context of kinetic equations; see also~\cite{Loh23}.

\subsection{Main results for divergence form equations}

We also establish new results for divergence-form equations.
In this context, the best known Schauder estimates are for operators with kernels satisfying \eqref{strong-ellipticity}; see Fall \cite{Fal}.

We improve substantially these estimates by allowing general elliptic operators \eqref{eq:wedenotedB} with kernels satisfying 
\begin{equation}\label{divergence-ellipticity0B}
r^{2s} \int_{B_{2r}(x)\setminus B_r(x)} K(x,z)\, dz\le \Lambda\qquad\text{for all}\quad x\in \R^n
\end{equation}
and 
\begin{equation}\label{divergence-ellipticity1B}
r^{2s-2}\inf_{e\in \S^{n-1}}\int_{B_{r}(x)}|e\cdot (x-z)|^{2} K(x,z)\, dz\ge \lambda>0\qquad\text{for all}\quad x\in \R^n.
\end{equation}

In order to obtain Schauder-type estimates, in addition to the uniform ellipticity assumptions \eqref{divergence-ellipticity0B}-\eqref{divergence-ellipticity1B} we need to assume some $C^\alpha$ regularity of the kernels in the $x$-variable.
More precisely, we assume 
\begin{equation}\label{reg-x-div-S}
\int_{B_{2\rho}(x)\setminus B_\rho(x)} \big|K(x+h,h+z)-K(x,z)\big|\, dz\leq M|h|^\alpha \rho^{-2s}
\end{equation}
for all $x, h\in \R^n$,   and $\rho > 0$. 
Besides the previous regularity in $x$ we also need to assume that the kernel, at small scales, is (quantitatively) \emph{almost even}:  
\begin{equation}\label{reg-x-div-even-S}
\int_{B_{2\rho}\setminus B_\rho} \big|K(x,x+y)-K(x,x-y)\big|\, dy  \leq M\rho^{\alpha-2s}
\end{equation}
for all $x\in \R^n$, and $\rho > 0$.

Finally, in some cases we will need to assume regularity in the $z$-variable as well, given by 
\begin{equation}\label{reg-x-div-y-S}
\int_{B_{2\rho}(x) \setminus B_\rho(x) } \big|K(x, h+z) - K(x, z)\big|\, dz  \le M |h|^\theta\rho^{-2s-\theta},
\end{equation}
(for some $\theta\in (0, 1]$) for all $h\in B_{\rho/2}$ and for all $x\in \R^n$ and $\rho > 0$.

The nonlocal Schauder estimates for general equations in divergence form read as follows:

\begin{thm}\label{thmC}
Let $s\in (0,1)$, $\alpha\in (0, 1]$, $\eps \in (0, \alpha)$, and let $\L$ be any operator of the form \eqref{eq:wedenotedB}-\eqref{eq:wedenotedB2} satisfying \eqref{divergence-ellipticity0B}-\eqref{divergence-ellipticity1B} uniformly in $x\in \R^n$, as well as
  \eqref{reg-x-div-S}-\eqref{reg-x-div-even-S} for some $M > 0$. 
 
Let $u\in C^{\beta}_{\rm loc}(B_1)\cap L^\infty(\R^n)$ be any weak solution of \eqref{linear}
with $f\in X$, and  
\begin{equation}\label{sjdgneorij2-S}
\beta:=\left\{ \begin{array}{rl}
1+\alpha & \textrm{if}\quad s>{\textstyle \frac12}, \vspace{1mm}\\
1+\alpha-\eps &  \textrm{if}\quad s={\textstyle \frac12}, \vspace{1mm} \\
2s+\alpha &  \textrm{if}\quad s<{\textstyle \frac12},
\end{array}\right.\qquad X := \left\{
\begin{array}{ll}
C^{\beta-2s}(B_1)& \quad\text{if}\quad \beta > 2s,\\
L^{\frac{n}{2s-\beta}}(B_1)& \quad\text{if}\quad \beta < 2s.
\end{array}
\right.
\end{equation}
  Assume in addition that $\beta\neq1$, $\beta\neq 2s$, and that \eqref{reg-x-div-y-S} holds if $\beta > 2s$, with $\theta = \beta-2s$. 
Then,
\[
\|u\|_{C^{\beta}(B_{1/2})} \le C\left(\|u\|_{L^\infty(\R^n)} + \|f\|_{X}\right)
\]
The constant $C$ depends only on $n$, $s$, $\alpha$, $\eps$, $\lambda$, $\Lambda$, and $M$. 
\end{thm}

Notice that conditions \eqref{reg-x-div-S}-\eqref{reg-x-div-even-S} are like a $C^\alpha$ regularity of the coefficients. 
Indeed, since $K(x, z) = K(z, x)$, we have that \eqref{reg-x-div-S}-\eqref{reg-x-div-even-S} hold automatically if one assumes the (stronger) pointwise condition
\[
\big|K(x+h, z+h) -K(x, z)\big| \le \frac{M|h|^\alpha}{|z-x|^{n+2s}},\qquad\text{for all}\quad x, z, h\in \R^n. 
\]
In general, though, our assumptions \eqref{reg-x-div-S}-\eqref{reg-x-div-even-S} allow for much more singular kernels. 
Furthermore, even for  operators  comparable to the fractional Laplacian, \eqref{strong-ellipticity}, the new assumptions \eqref{reg-x-div-S}-\eqref{reg-x-div-even-S} are the most general and comprehensible interpretation of $C^\alpha$ coefficients in the context of integro-differential operators.

We also establish a Cordes-Nirenberg type result in this context:

\begin{thm}\label{thmD}
Let $s\in (0,1)$, let $\beta\in (s, \min\{1,2s\})$, and let $\L$ be any operator of the form \eqref{eq:wedenotedB}-\eqref{eq:wedenotedB2} satisfying \eqref{divergence-ellipticity0B}-\eqref{divergence-ellipticity1B} uniformly in $x\in \R^n$, as well as 
\begin{equation}\label{reg-x-div-CN}
\int_{B_{2\rho}(x)\setminus B_\rho(x)} \big|K(x+h,h+z)-K(x,z)\big|dz\leq \eta \rho^{-2s}
\end{equation}
and
\begin{equation}\label{reg-x-div-even-CN}
\int_{B_{2\rho}\setminus B_\rho} \big|K(x,x+y)-K(x,x-y)\big|dy  \leq \eta\rho^{-2s}
\end{equation}
for all $x,h\in \R^n$, and $\rho > 0$.
Then, there exists $\eta>0$ small, depending only on $n$, $s$, $\beta$, $\lambda$, $\Lambda$, for which the following holds.
 
Let $u\in C^{\beta}_{\rm loc}(B_1)\cap L^\infty(\R^n)$ be any weak solution of \eqref{linear}
with $f\in L^{\frac{n}{2s-\beta}}(B_1)$.
Then,
\[
\|u\|_{C^{\beta}(B_{1/2})} \le C\left(\|u\|_{L^\infty(\R^n)} + \|f\|_{L^{\frac{n}{2s-\beta}}(B_1)}\right)
\]
The constant $C$ depends only on $n$, $s$, $\beta$, $\lambda$, $\Lambda$. 
\end{thm}

As said before, to our knowledge, these are the first Schauder and Cordes-Nirenberg type results for nonlocal elliptic operators in divergence form with kernels not satisfying \eqref{strong-ellipticity}.
Moreover, we expect our assumptions \eqref{reg-x-div-S}-\eqref{reg-x-div-even-S} to be essentially optimal for the Schauder estimate to hold.

\subsection{Organization of the paper}

In Section \ref{sec:prel_steps} we give some preliminary definitions and known results.
In Section \ref{sec3} we prove Theorem \ref{thmB}.
Then, in Section \ref{sec4} we prove Theorem \ref{thmC}.
In Section \ref{sec5} we prove Theorem \ref{thmD}, and finally in Section \ref{sec6} we give the proof of Theorem \ref{thmA}.
In addition, in Appendix~\ref{appA} we collect some useful results about H\"older spaces that we use throughout the paper.

\section{Preliminaries}\label{sec:prel_steps}

In this section we introduce the notation used throughout the work, as well as some preliminary results that will be useful in the following proofs. 

\subsection{Notation} \label{ssec:notation} Let us introduce some of the notation that will be used throughout the work. 

Since we work with kernels $K$ that are not necessarily absolutely continuous, from now on $K(dy)$ will denote a nonnegative Radon measure in $\R^n$.
Moreover, given a measure $K(dy)$, we denote by $K(a+r\, dy)$ the measure $\tilde K(dy)$ such that $\tilde K(B) = K(a+r B)$ for all $B\subset\R^n$ Borel. 

We also denote by $L^\infty_{\tau}(\R^n)$ the space of functions with bounded growth at infinity. Namely, given $\tau \ge 0$, we say that $w\in L^\infty_{\tau}(\R^n)$ if 
\[
\|w\|_{L^\infty_{\tau}(\R^n)} := \left\|\frac{w(x)}{1+|x|^\tau}\right\|_{L^\infty(\R^n)}<\infty.
\]

Finally, given $\alpha > 0$, we denote $C^{\alpha+}(\Omega) := \bigcup_{\eps >0}C^{\alpha+\eps}(\Omega)$. 

We next introduce the general class of elliptic operators $\GL$ (and its more regular counter-part, $\G_s(\lambda, \Lambda; \alpha)$); see Definition~\ref{defi:G}.

\subsection{Translation invariant operators} 

In the simplest case of operators without $x$-dependence, we have
\begin{equation}
\label{eq:Lu_nu}
\begin{split}
 \L u(x) &=   {\rm P.V.}\int_{\R^n} \big(u(x) - u(x+y)\big)K (dy) \\
 & =\frac12 \int_{\R^n} \big(2u(x) - u(x+y) - u(x-y)\big)K (dy),
 \end{split}
\end{equation}
where $K(dy)$ satisfies
\begin{equation}
\label{eq:nu_cond}
\text{$K\ge 0$ is symmetric, \quad $K(\{0\}) = 0,$}\quad \int_{\R^n} \min\{1, |y|^2\} \,K(dy) <\infty. 
\end{equation}

Operators like the ones above have a natural Fourier representation in terms of a Fourier multiplier $\A$. That is, in Fourier space, there exists a function $\A:\R^n\to \R$ such that
\[
\F(\L u)(\xi) = \A(\xi) \F(u)(\xi). 
\]

The weakest ellipticity assumption on $\L$ to obtain  interior regularity estimates of order $2s$ is given by a condition on their Fourier symbol, $\A$. 
Namely, we say that $\L$ is elliptic of order $2s$ with ellipticity constants $\tilde\Lambda$ and $\tilde\lambda$ in Fourier space if 
\begin{equation}
\label{eq:GSLO}
0 < \tilde \lambda |\xi|^{2s} \le \A(\xi) \le \tilde \Lambda |\xi|^{2s}\quad\text{for all}\quad \xi\in \R^n. 
\end{equation}

In physical space, such a condition is equivalent (see \cite{FR23}) to the existence of some $\lambda > 0$ and $\Lambda > 0$ such that 
\begin{equation}
\label{eq:Kellipt_gen_L} \index{Ellipticity!General operators}
r^{2s} \int_{B_{2r}\setminus B_r} K(dy) \le \Lambda\qquad\text{for all}\quad r  >0,
\end{equation}
and
\begin{equation}
\label{eq:Kellipt_gen_l}
r^{2s-2}\inf_{e\in \S^{n-1}}\int_{B_{r}}|e\cdot y|^{2} K(dy) \geq \lambda>0 \qquad\text{for all}\quad r  >0.
\end{equation}

On the other hand, given an operator as above, there is a natural way to assign it a H\"older-type regularity in the integral sense on its kernel, by considering the semi-norms: 
\begin{equation}\label{Calpha-assumption}
[\L]_\alpha = [K]_\alpha := \sup_{\rho > 0}  \sup_{x, x'\in B_{\rho/2}}\rho^{2s+\alpha}\int_{B_{2\rho}\setminus B_\rho}\frac{\big|K(x-dy)-K(x'-dy)\big|}{|x-x'|^{\alpha}},
\end{equation}

With the previous definitions, we introduce the most general class of translation invariant elliptic operators of order $2s$ under which one expects classical interior regularity estimates to equations of the form \eqref{linear}: 

\begin{defi}[General elliptic operator]
\label{defi:G}
Let $s\in (0, 1)$ and $\lambda, \Lambda > 0$. 
We define
\[
\GL :=\left\{\L : \begin{array}{l}
\text{$\L$ is an operator of the form \eqref{eq:Lu_nu}-\eqref{eq:nu_cond}}\\
\text{such that \eqref{eq:Kellipt_gen_L}-\eqref{eq:Kellipt_gen_l} hold}
\end{array}
 \right\}.
\]
Furthermore, given $\alpha\in(0, 1]$, we define the class
\[
\G_s(\lambda, \Lambda; \alpha) :=\big\{\L \in \GL : [\L]_\alpha < \infty\big\}
\]
of operators in $\GL$ with regular kernels, in the sense of \eqref{Calpha-assumption}. 
\end{defi}

\subsection{Some useful results} The following are known results that we recall here for the convenience of the reader. Most of them can be found (together with proofs) in \cite{FR23}. The first is a result regarding the boundedness of $\L  u$ when $u$ is regular.  

\begin{lem}[\cite{FR23}]
\label{lem:Lu}
Let $s\in (0, 1)$ and $\L\in \GL$. Let   $u\in C^{2s+\eps}(B_1)\cap L^\infty_{2s-\eps}(\R^n)$ for some $\eps > 0$. Then $\L u\in L^\infty_{\rm loc}(B_1)$ with 
\[
\|\L u\|_{L^\infty(B_{1/2})} \le C \Lambda \left(\|u\|_{C^{2s+\eps}(B_1)} + \|u\|_{L^{\infty}_{2s-\eps}(\R^n)} \right) 
\]
for some $C$ depending only on $n$, $s$, and $\eps$.
\end{lem}

Followed by a result on the regularity of $\L u$, again when $u$ is regular.

\begin{lem}[\cite{FR23}]
\label{lem:Lu_2}
Let $s\in (0, 1)$ and $\alpha\in (0, 1)$. Then:
\begin{enumerate}[leftmargin=*,label=(\roman*)]
\item \label{it:lem_Lu2_i} If $\L\in \G_s(\lambda, \Lambda; \alpha)$ (see Definition~\ref{defi:G}), then for any $u\in C^{2s+\alpha}(B_1)\cap L^\infty_{2s-\eps}(\R^n)$ with $\eps > 0$ we have  $\L u\in C_{\rm loc}^\alpha(B_{1})$ and
\[
[\L u]_{C^\alpha (B_{1/2})} \le C  \left(\Lambda \|u\|_{C^{2s+\alpha}(B_1)} + [\L]_\alpha \|u\|_{L^{\infty}_{2s-\eps}(\R^n)}\right),
\]
for some $C$ depending only on $n$, $s$, $\eps$, and $\alpha$. 
\item \label{it:lem_Lu2_ii} If $\L\in \GL$, then for any $u\in C^{2s+\alpha}(B_1)\cap C^\alpha(\R^n)$ we have $\L u\in C_{\rm loc}^\alpha(B_{1})$ and
\[
[\L u]_{C^\alpha (B_{1/2})} \le C \Lambda \left(\|u\|_{C^{2s+\alpha}(B_1)} +[u]_{C^{\alpha}(\R^n)}\right),
\]
with $C$ depending only on $n$, $s$, and $\alpha$.
 
\item \label{it:lem_Lu2_iii} If $\gamma\in (0, \alpha)$   and $\L\in \G_s(\lambda, \Lambda; \alpha-\gamma)$, then for any $u\in C^{2s+\alpha}(B_1)\cap C^\gamma(\R^n)$ we have $\L u\in C_{\rm loc}^\alpha(B_{1})$ and
\begin{equation}
\label{eq:rem_Lu22}
[\L u]_{C^\alpha (B_{1/2})} \le C  \left(\Lambda \|u\|_{C^{2s+\alpha}(B_1)} + [\L]_{\alpha-\gamma} \left(\|u\|_{L^\infty_{2s-\eps}(\R^n)} + [u]_{C^\gamma(\R^n)}\right)\right),
\end{equation}
for some $C$ depending only on $n$, $s$, $\eps$,  $\alpha$, and $\gamma$. 
\end{enumerate}
\end{lem}

\begin{rem}
\label{rem:nonneg_kernels}
 Lemmas~\ref{lem:Lu} and \ref{lem:Lu_2}   do not require the lower ellipticity assumptions, \eqref{eq:Kellipt_gen_l}, which can be seen  from the lack of dependence on $\lambda$; nor the nonnegativity of the kernel. We only need that $K$ is even and
\[
r^{2s}\int_{B_{2r}\setminus B_r} |K(dy)|  \le \Lambda < \infty\qquad\text{for all}\quad r > 0,
\]
as well as the regularity assumption \eqref{Calpha-assumption} in Lemma~\ref{lem:Lu_2}-\ref{it:lem_Lu2_i} and \ref{it:lem_Lu2_iii}.
\end{rem}

The following is a Liouville theorem for solutions to translation invariant nonlocal equations.

\begin{thm}[Liouville-type theorem, \cite{FR23}]
\label{thm:Liouville}
Let $s\in (0,1)$ and $\L\in \GL$. Let $u\in L^\infty_\beta(\R^n)$ for some $\beta\in [0, 2s)$ be a distributional or weak solution to 
\[
\L u = 0 \quad \text{in}\quad \R^n.
\]
Then, $u(x) = a\cdot x + b$, with $b = 0$ if $\beta < 1$.
\end{thm}

And we next recall the following interior regularity estimate for general translation invariant nonlocal equations:

\begin{thm}[\cite{FR23}]\label{thm-interior-linear-2}
Let $s\in (0,1)$ and $\L\in \G_s(\lambda, \Lambda; \alpha)$ for some $\alpha>0$ such that $2s+\alpha\notin\N$. Let $f\in C^\alpha(B_1)$, and let $u\in L^\infty_{2s-\eps}(\R^n)$ for some $\eps > 0$ be a distributional solution of
\[
\L u = f \quad\text{in}\quad B_1.
\]

Then, $u\in C_{\rm loc}^{2s+\alpha}(B_{1})$ with
\[
\|u\|_{C^{2s+\alpha}(B_{1/2})} \le C\left(\|u\|_{L^\infty_{2s-\eps}(\R^n)} + \|f\|_{C^\alpha(B_1)}\right)
\]
for some $C$ depending only on $n$, $s$, $\alpha$, $\eps$,  $\lambda$, $\Lambda$, and $[\L]_\alpha$. 
\end{thm}

We finally state the following technical lemma, that will be useful to establish regularity estimates in general settings.

\begin{lem}[\cite{FR23}]\label{lem-interior-blowup}
Let $\mu>0$ with $\mu\notin \mathbb Z$, let $\delta > 0$, and let us consider $\mathcal S: C^\mu(\R^n)\to\R_{\ge 0}$.
Then,
\begin{enumerate}[leftmargin=*, label=(\roman*)]
\item \label{it:lem_int_blowup_i} either we have
\[[u]_{C^\mu(B_{1/2})} \leq \delta [u]_{C^\mu(\R^n)} + C_\delta\big(\|u\|_{L^\infty(B_1)} + \mathcal S(u)\big)\]
for all $u\in C^\mu(\R^n)$, for some $C_\delta$ depending only on $\mu$, $\mathcal S$, and $\delta$,

\vspace{1mm}

\item \label{it:lem_int_blowup_ii} or there is a sequence $u_k\in C^\mu(\R^n)$, with 
\begin{equation}\label{lem-interior-blowup2}
\frac{\mathcal S(u_k)}{[u_k]_{C^\mu(B_{1/2})}} \longrightarrow 0,
\end{equation}
and there are $r_k\to0$, $x_k\in B_{1/2}$, such that if we define
\begin{equation}\label{lem-interior-blowup3}
v_k(x) := \frac{u_k(x_k+r_k x)}{r_k^\mu[u_k]_{C^\mu(\R^n)}},
\end{equation}
and we denote by $p_k$ the $\nu$-th order Taylor polynomial of $v_k$ at 0, then 
\[
\|v_k - p_k\|_{C^\nu(B_1)} > \frac{\delta}{2}
\]
for $k$ large enough, where $\nu=\lfloor \mu\rfloor$.
\end{enumerate}
\end{lem}

\subsection{Viscosity solutions}

In the following, we define the notion of viscosity (sub-/super-)solution to non-divergence-form equations \eqref{linear}.

\begin{defi}[Viscosity solutions] 
\label{defi:viscosity}
Let $s\in (0, 1)$,   let $\L$  be an operator  of the form \eqref{x-dependence-L-nu} with $\L_x \in \GL$ for all $x\in \R^n$, and let  $f\in C(\Omega)$.
\begin{itemize}[leftmargin=1cm]
\item We say that $u\in {\rm USC}({\Omega})\cap L^\infty_{2s-\eps}(\R^n)$ for some $\eps > 0$ is a \emph{viscosity subsolution} to \eqref{linear} in $\Omega$, and we denote   
\[\L (u, x) \le f(x)\quad  \textrm{in}\quad \Omega,\] 
if for any $x\in \Omega$ and any neighborhood of $x$ in $\Omega$, $N_x\subset \Omega$, and for any test function $\phi\in L^\infty_{2s-\eps}(\R^n)$ such that $\phi\in C^2(N_x)$, $\phi(x) = u(x)$, and $\phi \ge u$ in all of $\R^n$,   we have $\L_x \phi \le f(x)$. 

\vspace{1mm}

\item We say that $u\in {\rm LSC}(\Omega)\cap L^\infty_{2s-\eps}(\R^n)$ for some $\eps > 0$ is a \emph{viscosity supersolution} to \eqref{linear} in $\Omega$, and we denote   
\[\L (u, x) \ge f(x)\quad  \textrm{in}\quad \Omega,\]
if for any $x\in \Omega$ and any neighborhood of $x$ in $\Omega$, $N_x\subset \Omega$, and for   any   test function $\phi\in L^\infty_{2s-\eps}(\R^n)$ such that $\phi\in C^2(N_x)$, $\phi(x) = u(x)$, and $\phi \le u$ in all of $\R^n$,   we have $\L_x \phi \ge f(x)$. 

\vspace{1mm}

\item We say that $u\in C({\Omega})\cap L^\infty_{2s-\eps}(\R^n)$ for some $\eps > 0$ is a \emph{viscosity solution} to \eqref{linear} in $\Omega$, and we denote $\L(u, x) = f(x)$ in $\Omega$, if it is both a viscosity subsolution and supersolution. 
\end{itemize}
\end{defi}

\section{Schauder estimates in non-divergence form}
\label{sec3}

From now on, we denote as follows the operators in non-divergence form with $x$-dependence:
\begin{equation}\label{x-dependence-L-nu}
\begin{split}
\L(u,x)& := {\rm P.V.}\int_{\R^n}\big(u(x)-u(x+y)\big)K_x(dy) \\
& = \frac12 \int_{\R^n}\big(2u(x)-u(x+y)-u(x-y)\big)K_x(dy),
\end{split}
\end{equation}
(cf. \eqref{eq:Lu_nu}) where $(K_x)_{x\in \R^n}$ is a family of L\'evy measures satisfying \eqref{eq:nu_cond} and with ellipticity conditions \eqref{eq:Kellipt_gen_L}-\eqref{eq:Kellipt_gen_l} uniform in $x\in \R^n$. Namely, if for a given $x_\circ \in \R^n$ we denote $\L_{x_\circ}$ the translation invariant operator with L\'evy measure $K_{x_\circ}$ (i.e., the operator with ``frozen coefficients''), we consider  operators $\L$ of the form \eqref{x-dependence-L-nu} such that $\L_x\in \GL$ for all $x\in \R^n$.

The assumption of ``H\"older continuous coefficients'' in this context reads as follows:
\begin{equation}\label{x-dependence-L2}
\int_{B_{2\rho}\setminus B_\rho} \big|K_x(dy)-K_{x'}(dy)\big|  \leq M|x-x'|^\alpha \rho^{-2s},
\end{equation}
for all $x, x'\in \R^n$ and all $\rho > 0$, and for some $M > 0$. 
 
In some cases, we also need some regularity of each $K_x$ in the $y$-variable, namely 
\begin{equation}\label{x-dependence-L3}
\sup_{x\in\R^n} \, [K_x]_\alpha \leq M.
\end{equation}
(Recall \eqref{Calpha-assumption}.)

Under the previous assumptions, we then have the following a priori estimate (see Remark~\ref{rem:visc_x_dep}). It corresponds to Theorem~\ref{thmB} in the case $\gamma = 0$.

\begin{prop}\label{thm-interior-linear-x}
Let $s\in (0,1)$, and let $\L$ be an operator of the form \eqref{x-dependence-L-nu} with $\L_x\in \GL$ for all $x\in \R^n$. Suppose, moreover, that $\L$ satisfies the regularity assumptions \eqref{x-dependence-L2}-\eqref{x-dependence-L3} for some $\alpha\in(0,1]$  such that $2s+\alpha\notin\N$, and $M > 0$.  

Let $f\in C^\alpha(B_1)$, and let $u\in C_{\rm loc}^{2s+\alpha}(B_1) \cap L^\infty_{2s-\eps}(\R^n)$ for some $\eps > 0$ be any solution of \eqref{linear}.
Then, 
\[
\|u\|_{C^{2s+\alpha}(B_{1/2})} \le C\left(\|u\|_{L^\infty_{2s-\eps}(\R^n)} + \|f\|_{C^\alpha(B_1)}\right)
\]
for some $C$ depending only on $n$, $s$, $\alpha$, $\eps$, $\lambda$, $\Lambda$, and $M$. 
\end{prop}

Notice that we need the kernel to have the same degree of regularity (in the $y$ variable) as the right-hand side in order to gain $2s$ derivatives. 
Without this assumption, the previous estimate is false.

Still, for solutions satisfying a global regularity assumption of the type $u\in C^\gamma(\R^n)$, one expects to be able to remove (or at least, weaken) such an assumption. 
This is what we do in the following result, which holds under the weaker regularity assumption (in $x$) 
\begin{equation}\label{x-dependence-L4}
\sup_{\begin{subarray}{c} [\phi]_{C^\gamma(\R^n)}\leq 1 \\ \phi(0)=0 \end{subarray}} 
\left|\int_{B_{2\rho}\setminus B_\rho} \phi\, d(K_x - K_{x'})\right| \leq M|x-x'|^\alpha\rho^{\gamma-2s},
\end{equation}
for all $x,x'\in \R^n$ and all $\rho>0$. 
The expression on the right-hand side of \eqref{x-dependence-L4} can be understood as some kind of weighted Kantorovich norm measuring the distance between the measures $K_x$ and $K_{x'}$ (see \cite{Han92, Han99}). 
Notice that, in the following, no regularity in the $y$ variable is assumed when $\alpha\leq \gamma$.

\begin{prop}\label{prop-interior-linear-x}
Let $s\in (0,1)$,  $\gamma\in (0, 2s)$, $\alpha\in (0,1]$ such that $2s+\alpha\notin \N$, and let $\L$ be an operator of the form \eqref{x-dependence-L-nu}  such that $\L_x\in \GL$ for all $x\in \R^n$. 
Suppose, moreover, that 
 $\L$ satisfies \eqref{x-dependence-L4} for some $M > 0$. 
 In case $\alpha>\gamma$, assume in addition that $\sup_{x\in B_1} [K_x]_{\alpha-\gamma} \leq M$.

Let $f\in C^\alpha(B_1)$, and let $u\in C^{2s+\alpha}(B_1)\cap C^{\gamma}(\R^n)$ be any solution of \eqref{linear}.
Then
\[
\|u\|_{C^{2s+\alpha}(B_{1/2})} \le C\left(\|f\|_{C^\alpha(B_1)} + \|u\|_{C^\gamma(\R^n)}\right)
\]
for some $C$ depending only on $n$, $s$, $\alpha$, $\gamma$, $\lambda$, $\Lambda$, and $M$.
\end{prop}

Notice that the assumption \eqref{x-dependence-L4} is the minimal scale-invariant assumption that  ensures the property
\begin{equation}\label{nyi}
\big\|\L_{x_1} w-\L_{x_2}w\big\|_{L^\infty(B_{3/4})} \leq M|x_1-x_2|^\alpha \big(\|w\|_{C^{2s+\alpha}(B_1)}+\|w\|_{C^{\gamma}(\R^n)}\big),
\end{equation}
for all $w\in C^{2s+\alpha}(B_1)\cap C^{\gamma}(\R^n)$ (see Lemma~\ref{lem-sdhhh}), and thus we expect it to be the minimal assumption under which these Schauder-type estimates hold.

As explained in the Introduction, it is interesting to notice that, for $\alpha\leq \gamma$, \eqref{x-dependence-L4} allows for completely singular L\'evy measures.
For example, one could have operators of the type
\[\sum_{i=1}^n \big(-\partial_{\textbf{v}_i(x)}^2\big)^s,\]
where the directions $\textbf{v}_i$ are smooth functions of $x$ satisfying ${\rm det}(\textbf{v}_i)_i\neq0$.
Notice also that choosing $\alpha<\gamma$ allows us to consider H\"older continuous functions~$\textbf{v}_i$ (with exponent $\alpha/\gamma$).  Such operators are not covered, for example, by the stronger assumption \eqref{x-dependence-L2}.

\begin{rem} 
\label{rem:visc_x_dep}
The estimates we prove in this paper are \textit{a priori} estimates, in the sense that we assume the solutions to be $C^{2s+\alpha}$.
Still, using the theory of viscosity solutions, we expect one could actually show by a regularization procedure that the previous estimates are also valid for general viscosity solutions (cf. \cite{Kri13, Ser2, Fer23}). 
\end{rem}

\begin{proof}[Proof of Proposition \ref{thm-interior-linear-x}]
Let us divide the proof into three  steps.
\begin{steps}
\item\label{step:1nondiv} We start with an initial reduction. By a standard covering argument, we may prove the estimate in $B_{r_\circ/2}$ instead of $B_{1/2}$, where $r_\circ$ is a small fixed constant to be chosen later (depending only on $n$, $s$, $\alpha$, $\eps$, $\lambda$, $\Lambda$, and $M$).

In order to do that, we define $u_\circ(x) := u(r_\circ x)$, so that if $\L$ has kernel   $K_x(dy)$, and we consider $\L^{r_\circ}$ to be the operator with kernel $K^{r_\circ}_x(dy) = r_\circ^{ 2s} K_{r_\circ x}(r_\circ \, dy)$ (which has the same ellipticity constants as $K_x$), then $u_\circ$ satisfies
\[
\L^{r_\circ}(u_\circ, x) = f_\circ(x) := r_\circ^{2s} f(r_\circ x) \quad\text{in}\quad B_1.
\]

If we can now prove the desired estimate for $u_\circ$,
\[
\|u_\circ\|_{C^{2s+\alpha}(B_{1/2})} \le C\left(\|u_\circ\|_{L^\infty_{2s-\eps}(\R^n)} + \|f_\circ\|_{C^\alpha(B_1)}\right),
\]
we will be done, since 
\[
\begin{split}
r_\circ^{2s+\alpha} \|u\|_{C^{2s+\alpha}(B_{r_\circ/2})}& \le \|u_\circ\|_{C^{2s+\alpha}(B_{1/2})}\\
& \le C\left(\|u_\circ\|_{L^\infty_{2s-\eps}(\R^n)} + \|f_\circ\|_{C^\alpha(B_1)}\right)\\
&   \le C\left(\|u\|_{L^\infty_{2s-\eps}(\R^n)} + \|f\|_{C^\alpha(B_1)}\right).
\end{split}
\]
In particular, since $r_\circ$ will be fixed universally, this will yield the estimate in $B_{r_\circ/2}$, and after a covering (and rescaling) argument, in the whole $B_{1/2}$. 

The advantage of this rescaling is that, now, the new operator $\L^{r_\circ}$ satisfies a new condition \eqref{x-dependence-L2} with constants depending on $r_\circ$ (since we have ``expanded'' the space by a factor $r_\circ^{-1}$):
\begin{equation}\label{x-dependence-L2-2}
\sup_{\rho > 0}\sup_{x, x'\in \R^n} \rho^{2s}\int_{B_{2\rho}\setminus B_\rho} \frac{\big|K^{r_\circ}_x(dy)-K^{r_\circ}_{x'}(dy)\big|}{|x-x'|^\alpha}   \leq Mr^\alpha_\circ =:\delta,
\end{equation}
whereas condition \eqref{x-dependence-L3} remains the same (that is, with the same constant);
\begin{equation}\label{x-dependence-L3-2}
\sup_{x\in\R^n} \, [K^{r_\circ}_x]_\alpha \leq M.
\end{equation} 

 In all, up to replacing $u$, $\L$, and $f$,  by $u_\circ$, $\L^{r_\circ}$, and $f_\circ$, we can assume without loss of generality that condition \eqref{x-dependence-L2} holds with $M = \delta> 0$ arbitrarily small, but fixed, that will be chosen universally:
\begin{equation}\label{x-dependence-L2_delta}
\sup_{\rho > 0}\sup_{x, x'\in \R^n} \rho^{2s}\int_{B_{2\rho}\setminus B_\rho} \frac{\big|K_x(dy)-K_{x'}(dy)\big|}{|x-x'|^\alpha}   \leq \delta,
\end{equation}

\item \label{step:step2prove} Let $\L_0 \in \G_s(\lambda, \Lambda;\alpha)$ be the operator $\L_x$ corresponding to $x=0$.
Then, by the regularity estimates for translation invariant equations, Theorem~\ref{thm-interior-linear-2}, we have
\[
\|u\|_{C^{2s+\alpha}(B_{1/4})} \le C\left(\|u\|_{L^\infty_{2s-\eps}(\R^n)} + \|\L_0u\|_{C^\alpha(B_{1/2})}\right).
\]
Moreover, we also have
\[\|\L_0u\|_{C^\alpha(B_{1/2})} \leq \|f\|_{C^\alpha(B_{1/2})} + \|\L(u,\cdot) - \L_0u\|_{C^\alpha(B_{1/2})}.\]
We now claim that 
\begin{equation}\label{ngfhhghg} 
\|\L(u,\cdot) - \L_0u\|_{C^\alpha(B_{1/2})} \leq C\left(\delta\|u\|_{C^{2s+\alpha}(B_1)} + \|u\|_{L^\infty_{2s-\eps}(\R^n)}\right).
\end{equation}

For this, we define, for any $x\in B_{1/2}$ fixed, a new operator 
\[
\MR_x := \L_x - \L_0,
\]
so that $\L(u, x) - \L_0 u = \MR_x u$. Observe that $\MR_x$ has kernel $R_x(dy) := K_x(dy) - K_0(dy)$, but it does not belong to $\GL$ (since it is not positive). 

We have, nonetheless, that by assumption \eqref{x-dependence-L2_delta} applied with $x' = 0$,
\[
\rho^{2s}\int_{B_{2\rho\setminus B_\rho}} |R_x(dy)| \le \delta |x|^\alpha \le \delta\qquad\text{for all}\quad \rho > 0. 
\]
In particular, we can apply Lemma~\ref{lem:Lu} (see Remark~\ref{rem:nonneg_kernels}) and deduce 
\[
|\MR_x u(x)| \le C \delta \left( \|u\|_{C^{2s+\alpha}(B_1)} + \|u\|_{L^\infty_{2s-\eps}(\R^n)}\right),
\]
for any $x\in B_{1/2}$, which gives the $L^\infty$ bound on \eqref{ngfhhghg}. 

On the other hand, given any $x_1, x_2\in B_{1/2}$, we now want to bound the difference
\[
|\MR_{x_1} u(x_1) - \MR_{x_2} u(x_2)|\le |\MR_{x_1} u(x_1) - \MR_{x_1} u(x_2)|+|\MR_{x_1} u(x_2) - \MR_{x_2} u(x_2)|.
\]
For the first term, we  use Lemma~\ref{lem:Lu_2}-\ref{it:lem_Lu2_i} (together with Remark~\ref{rem:nonneg_kernels}) with operator $\MR_{x_1}$ fixed, to deduce
\[
|\MR_{x_1} u(x_1) - \MR_{x_1} u(x_2)| \le C|x_1-x_2|^\alpha\left(\delta \|u\|_{C^{2s+\alpha}(B_1)}+ 2M \|u\|_{L^\infty_{2s-\eps}(\R^n)}\right).
\]
We have also used here that, by assumption \eqref{x-dependence-L3}, $[R_{x_1}]_\alpha\le 2M$. 

For the second term, we can define yet another operator
\[
\tilde\MR_{x_1,x_2} := \MR_{x_1} - \MR_{x_2},
\]
which has kernel $\tilde R_{x_1,x_2} (y) := K_{x_1}(y) - K_{x_2}(y)$ and satisfies,  by \eqref{x-dependence-L2_delta},
\[
\rho^{2s}\int_{B_{2\rho}\setminus B_\rho}|\tilde R_{x_1,x_2}(dy)|  \le \delta |x_1-x_2|^\alpha. 
\]
Thus, we can  apply again Lemma~\ref{lem:Lu} and Remark~\ref{rem:nonneg_kernels} to deduce
\[
\begin{split}
|\MR_{x_1} u(x_2) - \MR_{x_2} u(x_2)| & = |\tilde \MR_{x_1,x_2} u(x_2)|\\
& \le C\delta |x_1-x_2|^\alpha (\|u\|_{C^{2s+\alpha}(B_1)}+\|u\|_{L^\infty_{2s-\eps}(\R^n)}). 
\end{split}
\]

Putting everything together, we have obtained \eqref{ngfhhghg}, where $\delta$ can still be chosen. 

\item \label{step3} In all, we have shown 
\[
\|u\|_{C^{2s+\alpha}(B_{1/4})} \le C\left(\delta\|u\|_{C^{2s+\alpha}(B_1)} + \|u\|_{L^\infty_{2s-\eps}(\R^n)} + \|f\|_{C^\alpha(B_{1/2})}\right)
\]
for all $u\in C^{2s+\alpha}(B_1)\cap L^\infty_{2s-\eps}(\R^n)$, where $\L(u, \cdot) = f$. Obtaining the desired estimate from the previous one is now a standard covering-type argument, after choosing $\delta$ universally small enough (depending only on $n$, $s$, $\alpha$, $\eps$, $\lambda$, $\Lambda$, and $M$); see, e.g., \cite[Lemma 2.27]{FR22} or \cite[Section 2.4]{FR23}. We obtain, 
\[
\|u\|_{C^{2s+\alpha}(B_{1/2})} \le C\left(\|u\|_{L^\infty_{2s-\eps}(\R^n)} + \|f\|_{C^\alpha(B_{1})}\right),
\]
for all $u\in C^{2s+\alpha}(B_1)\cap L^\infty_{2s-\eps}(\R^n)$, where $\L(u, \cdot) = f$.

From \ref{step:1nondiv}, this proves the estimate in a universally small ball $B_{r_\circ}$, and after a further covering and rescaling, this shows the desired result.
 \qedhere
\end{steps}
\end{proof}

To prove Proposition  \ref{prop-interior-linear-x} we need the following.

\begin{lem}\label{lem-sdhhh}
Let $s\in(0,1)$,   $\alpha\in(0,1]$, and $\gamma\in(0,2s)$.
Let $\L_1,\L_2\in \GL$, with L\'evy measures $K_1,K_2$, be such that
\begin{equation}\label{iserbgieub}
\sup_{\begin{subarray}{c} [\phi]_{C^\gamma(\R^n)}\leq 1 \\ \phi(0)=0 \end{subarray}} 
\left|\int_{B_{2\rho}\setminus B_\rho} \phi \, d(K_1 - K_2)\right| \leq \theta \rho^{\gamma-2s}\qquad\text{for all}\quad \rho > 0,
\end{equation}
for some $\theta > 0$. 
 In case $\alpha>\gamma$, assume in addition that $[\L_i]_{\alpha-\gamma} \leq M$ for $i= 1,2$, and some $M > 0$. 

Then, if $w\in C^{2s+\alpha}(B_1)\cap C^{\gamma}(\R^n)$ we have
\begin{equation}
\label{eq:Linfty_nondivform}
\big\|(\L_1-\L_2)w\big\|_{L^\infty(B_{1/2})} \leq C 
\theta \big(\|w\|_{C^{2s+\alpha}(B_1)}+[w]_{C^{\gamma}(\R^n)}\big)
\end{equation}
as well as  
\begin{equation}
\label{eq:Calpha_nondivform}
\big[(\L_1-\L_2)w\big]_{C^\alpha(B_{1/2})}  \le 
\left\{
\begin{array}{ll}
  C 
\theta \big( \|w\|_{C^{2s+\alpha}(B_1)}+[w]_{C^{\gamma}(\R^n)}\big)& \quad\text{if}\quad \alpha \le \gamma\\[0.2cm]
 C 
\big(\theta \|w\|_{C^{2s+\alpha}(B_1)}+M [w]_{C^{\gamma}(\R^n)}\big)& \quad\text{if}\quad \alpha > \gamma.
\end{array}
\right.
\end{equation}
The constant $C$ depends only on $n$, $s$, $\alpha$, and $\gamma$.
\end{lem}

\begin{proof}
We denote 
\[
\MR := \L_1-\L_2.
\]
We divide the proof into four steps.
\begin{steps}
 \item \label{step:firsttodo}
We first prove an $L^\infty$ bound for $\MR w$, \eqref{eq:Linfty_nondivform}.
 Dividing by a constant if necessary, we may assume $\|w\|_{C^{2s+\alpha}(B_1)}+[w]_{C^\gamma(\R^n)}\leq 1$, and by taking $\alpha$ smaller if necessary, we may also assume $2s+\alpha < 2$ and $2s+\alpha\neq 1$.

Let $x_\circ\in B_{1/2}$.
Taking $\phi(y) = 2w(x_\circ)-w(x_\circ+y)-w(x_\circ-y)$ in \eqref{iserbgieub}, we find that for any  $\rho>0$ (using that $[w]_{C^\gamma(\R^n)}\le 1$), 
\begin{equation}
\label{eq:firstineqlem}
\left|\int_{B_{2\rho}\setminus B_\rho} \hspace{-0.35cm} \big(2w(x_\circ)-w(x_\circ+y)-w(x_\circ-y)\big)  (K_1 - K_2)(dy)\right| \leq 4\theta\rho^{\gamma-2s}.
\end{equation}
On the other hand, since $w\in C^{2s+\alpha}(B_1)$ with $\|w\|_{C^{2s+\alpha}(B_1)}\le 1$, we have (see  Lemma~\ref{it:H7_gen})
\[
\begin{split}
\left\|\phi\right\|_{L^\infty(B_{2\rho})}& \le C \rho^{2s+\alpha},\\
\|\phi\|_{C^{2s+\alpha}(B_{2\rho})}& \le 4,
\end{split}
\] 
for any $\rho <\frac14$. 
By interpolation inequality (Lemma~\ref{lem:interp_mult}) applied to the function $\phi$ in $B_{2\rho}$ we obtain 
\[\left[ \frac{\phi}{\rho^{2s+\alpha-\gamma}} \right]_{C^\gamma(B_{2\rho}\setminus B_\rho)} \leq C,\]
to get, by \eqref{iserbgieub},
\begin{equation}
\label{eq:secineqlem}
\left|\int_{B_{2\rho}\setminus B_\rho} \hspace{-0.35cm} \big(2w(x_\circ)-w(x_\circ+y)-w(x_\circ-y)\big) (K_1 - K_2)(dy)\right| \leq C\theta\rho^{\alpha}.
\end{equation}
Using the first inequality, \eqref{eq:firstineqlem}, for $\rho=2^{k}$, $k=-1,0,1,2,3,...$, and the second one, \eqref{eq:secineqlem}, for $k=-2,-3,...$,  and summing a geometric series, we deduce that
\[
\left|\int_{\R^n} \big(2w(x_\circ)-w(x_\circ+y)-w(x_\circ-y)\big)\,  (K_1 - K_2)(dy)\right| \leq C\theta.
\]
This proves the $L^\infty$ bound for $\MR w$.
\item 
Let us now show the $C^\alpha$ bounds, \eqref{eq:Calpha_nondivform}.
For this, we  split $w=u_1+u_2$, with $u_1:=\eta w$ and $\eta\in C^\infty_c(B_1)$ such that $\eta \ge 0$, $\eta \equiv 0$ in $\R^n\setminus B_{3/4}$ and $\eta \equiv 1$ in $B_{2/3}$.
Notice that $\|u_1\|_{C^{2s+\alpha}(\R^n)}\le C\|w\|_{C^{2s+\alpha}(B_1)}$ and $[u_2]_{C^\gamma(\R^n)} \leq C[w]_{C^\gamma(\R^n)}$.

We prove first the bound for $\MR u_1$ in the case $2s+\alpha \le 2$. Dividing by a constant if necessary, we assume $\|u_1\|_{C^{2s+\alpha}(\R^n)}\leq 1$. Observe that the $L^\infty$ bound for $\MR u_1$ follows by \ref{step:firsttodo}. We now want to bound the $C^\alpha$ seminorm, and more precisely, we will bound 
\begin{equation}
\label{eq:sameasbefore222}
\big|\MR u_1(x_\circ) + \MR u_1(-x_\circ) - 2 \MR u_1(0) \big| \le C \theta r^\alpha.
\end{equation}

Let $x_\circ\in B_{1/2}$ be fixed, and $r := |x_\circ|$. We split
\[
\begin{split}
\MR u_1(x_\circ) & = \frac12 \int_{B_r}\big(2u_1(x_\circ)-u_1(x_\circ+y)-u_1(x_\circ-y)\big)(K_1-K_2)(dy)\\
& \hspace{-5mm} +  \frac12 \int_{\R^n\setminus B_r}\big(2u_1(x_\circ)-u_1(x_\circ+y)-u_1(x_\circ-y)\big)(K_1-K_2)(dy).
\end{split}
\]
Then, since $2s+\alpha \le 2$ and by interpolation as in the previous step one can show
\begin{equation}
\label{eq:theexp}
\begin{split}
\big[u_1(x_\circ+\cdot\,)+u_1(x_\circ-\cdot\,)-2u_1(x_\circ)\big]_{C^\gamma(B_{2\rho}\setminus B_\rho)} &\le  C\rho^{2s+\alpha-\gamma},\\
\big[u_1(x_\circ\pm \cdot\,)+u_1(-x_\circ\pm \cdot\,)-2u_1(\pm \,\cdot\,)\big]_{C^\gamma(\R^n)} &\le  Cr^{2s+\alpha-\gamma}.
\end{split}
\end{equation}
If we now denote $\delta^2_h v(x)$ the second order centered increments, 
\[
  \delta^2_h v(x) = \frac{v(x+h)+v(x-h)}{2} - v(x),
\]
and 
\[
\MR u_1(x_\circ) + \MR u_1(-x_\circ) - 2 \MR u_1(0)  = 2\int_{\R^n} \phi(x_\circ, y)\, (K_1-K_2)(dy)
\]
with 
\[
\phi(x_\circ, y) := \delta^2_{y} u_1(x_\circ) + \delta^2_{y} u_1(-x_\circ) -2 \delta^2_{y} u_1(0),
\]
then the expressions \eqref{eq:theexp} imply, by definition \eqref{iserbgieub},
\[
\left|\int_{B_{2\rho}\setminus B_\rho} \phi(x_\circ,y)(K_1-K_2)(dy) \right| \le C\theta \min\{\rho^\alpha, \rho^{\gamma-2s}r^{2s+\alpha-\gamma}\}.
\]
Summing first for $\rho\in (0, r)$ (and taking the first argument in the $\min$), and summing then for $\rho > r$ (and taking the second argument in the $\min$) we obtain \eqref{eq:sameasbefore222}.

\item Let us now show \eqref{eq:sameasbefore222} in the case  $2<2s+\alpha<3$. By Lemma~\ref{lem:A_imp_2}-\ref{it:A:3} and the interpolation in Lemma~\ref{lem:interp_mult}   we have, on the one hand,
\begin{equation}
\label{eq:2salpha1_2}
\begin{split}
& \big[u_1(x_\circ+\cdot\,)\hspace{-0.5mm}+\hspace{-0.5mm}u_1(x_\circ-\cdot\,)\hspace{-0.5mm}-\hspace{-0.5mm}2u_1(x_\circ)\hspace{-0.5mm}-\hspace{-0.5mm}u_1(\,\cdot\,)\hspace{-0.5mm}-\hspace{-0.5mm}u_1(-\,\cdot\,)\hspace{-0.5mm}+\hspace{-0.5mm}2u_1(0)\big]_{C^\gamma(B_{2\rho}\setminus B_\rho)} \leq \\
& \qquad \le C\rho^{2-\gamma} r^{2s+\alpha-2} + C \rho^{2\left(1-\frac{\gamma}{2s+\alpha}\right)} r^{(2s+\alpha-2)\left(1-\frac{\gamma}{2s+\alpha}\right)}=:I_1.
\end{split}
\end{equation}
On the other hand, we want to find an appropriate bound for 
\begin{equation}
\label{eq:2salpha2_2}
\begin{split}
& \big[u_1(x_\circ+\cdot\,)\hspace{-0.71mm}+\hspace{-0.71mm}u_1(-x_\circ+\cdot\,)\hspace{-0.71mm}-\hspace{-0.71mm}2u_1(\,\cdot\,)\hspace{-0.71mm}-\hspace{-0.71mm}u_1(x_\circ)\hspace{-0.71mm}-\hspace{-0.71mm}u_1(-x_\circ)\hspace{-0.71mm}+\hspace{-0.71mm}2u_1(0)\big]_{C^\gamma(B_{2\rho}\setminus B_\rho)} =\\
& \qquad\le 2[\delta_{x_\circ}^2 u_1(\,\cdot\,)]_{C^\gamma(B_{2\rho})}.
\end{split}
\end{equation}
We do so by separating into three possible cases according to the value of~$\gamma$:
\begin{itemize}
\item If $\gamma \le 2s+\alpha -2 < 1$, then we have 
\[
\frac{|\delta_{x_\circ}^2 u_1(y) - \delta_{x_\circ}^2 u_1(y')|}{|y-y'|^\gamma} \le C |x_\circ|^2 |y-y'|^{2s+\alpha-2-\gamma}\le Cr^2\rho^{2s+\alpha-2-\gamma},
\]
for all $y, y'\in B_{2\rho}$, where we have used that $2s+\alpha-2-\gamma \ge 0$, and $[u_1]_{C^{2s+\alpha}(\R^n)}\le 1$, together with Lemma~\ref{lem:A_imp_2}-\ref{it:A:3}.
\item If $2s+\alpha-2 < \gamma \le 1$, then we can use again Lemma~\ref{lem:A_imp_2}-\ref{it:A:3} but now with $t$ such that $t(2s+\alpha -2) + (1-t) = \gamma$, where $t\in [0, 1)$ by assumption on $\gamma$, to obtain 
\[
\frac{|\delta_{x_\circ}^2 u_1(y) - \delta_{x_\circ}^2 u_1(y')|}{|y-y'|^\gamma} \le C |x_\circ|^{2s+\alpha-\gamma} \le Cr^{2s+\alpha-\gamma}
\]
for all $y, y'\in B_{2\rho}$. 

\item Finally, if $\gamma > 1$ we use the second equation in Lemma~\ref{lem:A_imp_2}-\ref{it:A:1} applied to $\nabla u_1$ and with $t = \gamma-1$ to derive (since $[\nabla u_1]_{C^{2s+\alpha-1}(\R^n)}\le C$)
\[
\frac{|\delta_{x_\circ}^2 \nabla u_1(y) - \delta_{x_\circ}^2 \nabla u_1(y')|}{|y-y'|^{\gamma-1}} \le C |x_\circ|^{2s+\alpha-\gamma} \le Cr^{2s+\alpha-\gamma}
\]
for all $y, y'\in B_{2\rho}$. 
\end{itemize}

We thus have a bound for the expression \eqref{eq:2salpha2_2} of the form 
\begin{equation}
\label{eq:2salpha2_3}
[\delta_{x_\circ}^2 u_1(\,\cdot\,)]_{C^\gamma(B_{2\rho})}\le I_2^\gamma := \left\{\begin{array}{ll}
C r^2 \rho^{2s+\alpha-2-\gamma} & \text{if} \ \gamma\le 2s+\alpha -2\\
 C r^{2s+\alpha-\gamma} & \text{if} \ \gamma > 2s+\alpha -2. 
\end{array}
\right.
\end{equation}

In particular, \eqref{eq:2salpha1_2}-\eqref{eq:2salpha2_3} imply now 
\[
\left|\int_{B_{2\rho}\setminus B_\rho} \phi(x_\circ,y)(K_1-K_2)(dy) \right| \le C\theta \rho^{\gamma-2s}\min\{I_1,I_2^\gamma\}.
\]
We now sum as before for $\rho = 2^k$ and $k\in \Z$, separating between $\rho < r$ and $\rho > r$. That is, we consider 
\begin{equation}
\label{eq:wecando}
\big|\MR u_1(x_\circ) + \MR u_1(-x_\circ) - 2 \MR u_1(0) \big| \le C\theta \sum_{\substack{\rho = 2^k\\ \rho < r}}\rho^{\gamma-2s} I_1+ C\theta \sum_{\substack{\rho = 2^k\\ \rho \ge r}}\rho^{\gamma-2s} I_2^\gamma.
\end{equation}

For the first term in the sum, we observe that the exponents of $\rho$ are positive, since they are 
\[
2-\gamma + \gamma - 2s > 0
\]
and 
\[
2\left(1-\frac{\gamma}{2s+\alpha}\right) + \gamma -2s > 0 ,
\]
where we are using that $2s+\alpha > 2$ and $2 > 2s$. In the second term of \eqref{eq:wecando}, the exponent of $\rho$ is negative for any $\gamma$, since for $\gamma \le 2s+ \alpha-2$ it is $\alpha -2$, and for $\gamma > 2s+\alpha-2$ it is $\gamma -2s$. We can therefore perform the sum in \eqref{eq:wecando} and obtain \eqref{eq:sameasbefore222} also in this case. 

 Repeating around any point in $B_{1/2}$ and thanks to the $L^\infty$ bound for $\MR u_1$ 
we get 
\[
\|\MR u_1\|_{C^\alpha(B_{1/2})}\le C\theta.
\]

\item Finally, the bound for $u_2$ 
\[
[\MR u_2]_{C^\alpha(B_{1/2})}\le [\MR u_2]_{C^\gamma (B_{1/2})}\le C\theta  [w]_{C^\gamma(\R^n)}\quad\text{if}\quad \alpha \le \gamma
\]
follows directly from the expression \eqref{eq:firstineqlem} (with $w = u_2$), and summing for $\rho = 2^k$, for $k = -1,0,1,2,\dots$ (since $u_2 \equiv 0$ in $B_{1/2}$). On the other hand, the bound 
\[
[\MR u_2]_{C^\alpha(B_{1/2})} \le CM [w]_{C^\gamma(\R^n)}\quad\text{if}\quad \alpha > \gamma
\]
follows separately for $\L_1$ and $\L_2$ by using \eqref{eq:rem_Lu22} (and a rescaling and covering argument)\qedhere
\end{steps}
\end{proof}

Thanks to the previous Lemma, we can now prove  Proposition  \ref{prop-interior-linear-x}.

\begin{proof}[Proof of Proposition  \ref{prop-interior-linear-x}]
The proof is essentially the same as that of Theorem~\ref{thm-interior-linear-x}, however \eqref{ngfhhghg} needs to be replaced by
\begin{equation}\label{ngfhhghg2} 
\|\L(u,\cdot) - \L_0u\|_{C^\alpha(B_1)} \leq C\left(\delta\|u\|_{C^{2s+\alpha}(B_2)} + [u]_{C^\gamma(\R^n)}\right).
\end{equation}
The proof of \eqref{ngfhhghg2} follows exactly as the proof of \eqref{ngfhhghg} by replacing the use of Lemmas~\ref{lem:Lu} and \ref{lem:Lu_2}-\ref{it:lem_Lu2_i} by \eqref{eq:Linfty_nondivform} and \eqref{eq:Calpha_nondivform} in Lemma~\ref{lem-sdhhh}, respectively, with $\theta = \delta$ small. The fact that $\theta$ can be taken to be small is for the exact same reason as in the proof of Theorem~\ref{thm-interior-linear-x}.
\end{proof}

Finally, we have the:

\begin{proof}[Proof of Theorem \ref{thmB}]
The result is the combination of Propositions \ref{thm-interior-linear-x} (corresponding to the case $\gamma=0$) and \ref{prop-interior-linear-x} (case $\gamma>0$).
\end{proof}

\section{Schauder estimates in divergence form} 
\label{sec4}

From now on, nonlocal operators in divergence form with $x$-dependence (with kernels that are not necessarily absolutely continuous) are denoted as follows:
\begin{equation}\label{divergence-form0}
\begin{split}
\L (u,x) & = {\rm P.V.} \int_{\R^n}\big(u(x)-u(z)\big)K(x,dz) \\
 &  = {\rm P.V.} \int_{\R^n}\big(u(x)-u(x+y)\big)K(x,x+dy) ,
\end{split}
\end{equation}
where $(K(x, \cdot))_{x\in \R^n}$ is a family of measures  in $\R^n$ that satisfies the uniform ellipticity conditions 
\begin{equation}\label{divergence-ellipticity0}
r^{2s} \int_{B_{2r}(x)\setminus B_r(x)} K(x,dz)\le \Lambda\qquad\text{for all}\quad x\in \R^n
\end{equation}
and 
\begin{equation}\label{divergence-ellipticity1}
r^{2s-2}\inf_{e\in \S^{n-1}}\int_{B_{r}(x)}|e\cdot (x-z)|^{2} K(x,dz)\ge \lambda>0\qquad\text{for all}\quad x\in \R^n,
\end{equation}
as well as symmetry in the two variables, in the sense that
\begin{equation}\label{divergence-form1}
\begin{split}
\int_{A}\int_B  K(x,dz)&\, dx  =\int_B\int_A K(x,dz)\, dx\\
& \quad \textrm{for all}\  A,B\subset \R^n\ \text{Borel, such that} \ A\cap B = \varnothing.
\end{split}
\end{equation}
Observe that, when $(K(x, \cdot))_{x\in \R^n}$ are absolutely continuous with respect to the Lebesgue measure, then \eqref{divergence-form1} reads as
\[
K(x, z) = K(z, x)\qquad \text{for a.e.}\quad  (x, z)\in \R^n\times \R^n.
\]

Equations of the type
\begin{equation}
\label{eq:div_weak}\L(u,x)=f(x)\quad\textrm{in}\quad \Omega
\end{equation}
have a natural weak formulation:
 
\begin{defi} Let $s\in (0, 1)$, and let $\L(\cdot, x)$ be of the form \eqref{divergence-form0}-\eqref{divergence-ellipticity0}-\eqref{divergence-ellipticity1}-\eqref{divergence-form1}. Let $\Omega\subset \R^n$ be any bounded domain, and let $f\in L^p(\Omega)$ for some $p \ge \frac{2n}{n+2s}$ and $n > 2s$. Let $u$ be such that
\[
\iint_{\R^n\times \R^n\setminus(\Omega^c\times \Omega^c)} \left(u(x) - u(z)\right)^2 K(x, dz)\, dx<\infty.
\]
We say that $u$ is a \emph{weak solution} of \eqref{eq:div_weak} if 
\[
\begin{split}\frac12 \int_{\R^n}  \int_{\R^n}   \big(u(x)-u(z)\big)\big(\eta(x)-\eta(z)\big)K( x,dz)& \, dx  =\int_{\R^n}  f\eta 
\end{split}\]
for all $\eta\in C^\infty_c(\Omega)$.

We say that $u$ is a \emph{weak supersolution} of \eqref{eq:div_weak} (resp. \emph{weak subsolution} of \eqref{eq:div_weak}) and we denote it $\L(u, x) \ge f(x)$ in $\Omega$ (resp. $\L(u, x) \le f(x)$ in $\Omega$) if 
\[
\begin{split}\frac12 \int_{\R^n}  \int_{\R^n}   \big(u(x)-u(z)\big)\big(\eta(x)-\eta(z)\big)K( x,dz)& \, dx  \underset{\left(\text{resp. $\le$}\right)}{\ge} \int_{\R^n}  f\eta 
\end{split}\]
for all $\eta\in C^\infty_c(\Omega)$ with $\eta\ge 0$.
\end{defi}

It is important to notice that, in case of divergence-form equations \eqref{divergence-form0}, one cannot symmetrize the operator and write it in terms of  a second-order incremental quotient $2u(x)-u(x+y)-u(x-y)$.
In particular, when $s\geq\frac12$ one cannot evaluate in general $\L(u,x)$ pointwise\footnote{This also happens for operators in divergence form ${\rm div}(A(x) \nabla u)$ in the local case $s = 1$.}  even for smooth functions $u\in C^\infty_c(\Omega)$.

In order to obtain Schauder-type estimates, in addition to the uniform ellipticity assumptions \eqref{divergence-ellipticity0}-\eqref{divergence-ellipticity1} we need to assume some $C^\alpha$ regularity of the kernels in the $x$-variable.
More precisely, we assume 
\begin{equation}\label{reg-x-div}
\int_{B_{2\rho}(x)\setminus B_\rho(x)} \big|K(x+h,h+dz)-K(x,dz)\big|\leq M|h|^\alpha \rho^{-2s}
\end{equation}
for all $x, h\in \R^n$,   and $\rho > 0$.
We also need to assume that the kernel, at small scales, is (quantitatively) \emph{almost even}:  
\begin{equation}\label{reg-x-div-even}
\int_{B_{2\rho}\setminus B_\rho} \big|K(x,x+dy)-K(x,x-dy)\big|  \leq M\rho^{\alpha-2s}
\end{equation}
for all $x\in \R^n$,    and $\rho > 0$.

On the other hand, in some cases we will need to assume regularity in the $y$-variable as well, given by 
\begin{equation}\label{reg-x-div-y}
\int_{B_{2\rho}(x) \setminus B_\rho(x) } \big|K(x, h+dz) - K(x, dz)\big|  \le M |h|^\theta\rho^{-2s-\theta},
\end{equation}
(for some $\theta\in (0, 1]$) for all $h\in B_{\rho/2}$ and for all $x\in \R^n$ and $\rho > 0$. Observe that the previous condition is equivalent to asking that, using the notation in \eqref{Calpha-assumption}, $\sup_{x\in \R^n} [K(x, x+\,\cdot\,)]_\theta\le M$.

The interior Schauder estimates for nonlocal divergence form equations are the following (which is exactly Theorem \ref{thmC} in the new notation):

\begin{thm}\label{thm-interior-linear-x-div}
Let $s\in (0,1)$, $\alpha\in (0, 1]$, $\eps \in (0, \alpha)$, and let $\L$ be an operator of the form \eqref{divergence-form0}-\eqref{divergence-form1}, with kernels satisfying the ellipticity conditions \eqref{divergence-ellipticity0}-\eqref{divergence-ellipticity1}, and \eqref{reg-x-div}-\eqref{reg-x-div-even} for some  $M > 0$. 
 
Let $u\in C^{\beta}_{\rm loc}(B_1)\cap L^\infty_{2s-\eps}(\R^n)$ be a weak solution of \eqref{linear},
with $f\in X$, and  
\begin{equation}\label{sjdgneorij2}
\beta:=\left\{ \begin{array}{rl}
1+\alpha & \textrm{if}\quad s>{\textstyle \frac12}, \vspace{1mm}\\
1+\alpha-\eps &  \textrm{if}\quad s={\textstyle \frac12}, \vspace{1mm} \\
2s+\alpha &  \textrm{if}\quad s<{\textstyle \frac12},
\end{array}\right.\qquad X := \left\{
\begin{array}{ll}
C^{\beta-2s}(B_1)& \quad\text{if}\quad \beta > 2s,\\
L^{\frac{n}{2s-\beta}}(B_1)& \quad\text{if}\quad \beta < 2s.
\end{array}
\right.
\end{equation}
  Assume in addition that $\beta\neq1$, $\beta\neq 2s$, and that \eqref{reg-x-div-y} holds if $\beta > 2s$, with $\theta = \beta-2s$. 
Then,
\[
\|u\|_{C^{\beta}(B_{1/2})} \le C\left(\|u\|_{L^\infty_{2s-\eps}(\R^n)} + \|f\|_{X}\right)
\]
The constant $C$ depends only on $n$, $s$, $\alpha$, $\eps$, $\lambda$, $\Lambda$, and $M$. 
\end{thm}

 The strategy to prove this result will be different in cases $\beta > 2s$ and $\beta < 2s$. When $\beta > 2s$, we will treat $\L$ as a (nonsymmetric) operator in non-divergence form, and argue as in the proof of Proposition~\ref{thm-interior-linear-x}. Instead, when $\beta < 2s$, the proof will be by contradiction and blow-up \emph{\`a la} L. Simon \cite{Sim97}, similarly to the proof of Theorem~\ref{thmA} that we provide later on.

\begin{proof}[Proof of Theorem \ref{thm-interior-linear-x-div} (and \ref{thmC}) in case $\beta>2s$]
When $\beta>2s$ (that is,  $\theta := \beta - 2s > 0$), the operator $\L$ can be evaluated pointwise on smooth functions $u$, and thus it can be seen as a (nonsymmetric) equation in non-divergence form.
Using this, we can follow the strategy of the proof of Theorem~\ref{thm-interior-linear-x} above. We divide the proof into  six steps. 

\begin{steps}
\item \label{step:1div} 
As in \ref{step:1nondiv} of the proof of Theorem~\ref{thm-interior-linear-x}, we start with an initial reduction wherein, up to considering the rescalings $u_\circ(x) := u(r_\circ x)$ and $\L^{r_\circ}$ (with kernel  $K^{r_\circ} (x, dz) = r_\circ^{ 2s}K(r_\circ x, r_\circ \, dz)$, where $K$ is the kernel of $\L$), we can assume that the operator $\L$ satisfies the ellipticity conditions \eqref{divergence-ellipticity0}-\eqref{divergence-ellipticity1}, has regularity in $y$ given by \eqref{reg-x-div-y} for some $M > 0$, but conditions \eqref{reg-x-div} and \eqref{reg-x-div-even} now become
\begin{equation}\label{reg-x-div-2}
\int_{B_{2\rho}(x)\setminus B_\rho(x)} \big|K(x+h,h+dz)-K(x,dz)\big|  \leq Mr_\circ^\alpha |h|^\alpha \rho^{-2s} =: \delta |h|^\alpha \rho^{-2s}
\end{equation}
for all $x, h\in \R^n$,   and $\rho > 0$; and 
\begin{equation}\label{reg-x-div-even-2}
\int_{B_{2\rho}\setminus B_\rho} \big|K(x,x+dy)-K(x,x-dy)\big|  \leq Mr_\circ^\alpha \rho^{\alpha-2s} =: \delta \rho^{\alpha -2s}
\end{equation}
for all $x\in \R^n$,    and $\rho > 0$; for some $\delta > 0$ a small universal constant to be chosen.

\item  Let us denote by $\L^e$ and $\L^o$ respectively the even and odd parts of the operator $\L$. Namely, we have 
\[
 \L^e(u, x) = {\rm P.V.} \int_{\R^n}\big(u(x)-u(x+y)\big)K^e(x,dy),
 \]
 where
 \[
  K^e(x, dy) := \frac{K(x, x+dy) + K(x, x-dy)}{2}
\]
is an even kernel, in the sense that $K^e(x, dy) = K^e(x, -dy)$ (i.e., $K^e(x, dy)$ is a symmetric measure); and 
\[
 \L^o(u, x) = {\rm P.V.} \int_{\R^n}\big(u(x)-u(x+y)\big)K^o(x,dy),
 \]
 where
 \[
  K^o(x, dy) := \frac{K(x, x+dy) - K(x, x-dy)}{2}
\]
is an odd kernel, in the sense that $K^o(x, dy) = -K^o(x, -dy)$. With these definitions, we have that $\L(u, x) = \L^e(u, x) + \L^o(u, x)$. Notice, moreover, that in the case of $\L^e$ we can symmetrize its expression as 
\[
 \L^e(u, x) = \frac12 \int_{\R^n}\big(2u(x)-u(x+y)-u(x-y)\big)K^e(x,dy),
 \]
which is now well-defined in $B_1$ (even without the principal value) because $u\in C^\beta_{\rm loc}(B_1)$ with $\beta > 2s$; cf. Lemma~\ref{lem:Lu}. In fact, the operator $\L^e$ is an operator in non-divergence form like the ones in Theorem~\ref{thm-interior-linear-x}, where the ellipticity conditions are satisfied thanks to \eqref{divergence-ellipticity0}-\eqref{divergence-ellipticity1} and linearity, $\sup_{x\in\R^n} [K^e(x, \cdot)]_\theta \le M$  holds  for $K^e$ thanks to \eqref{reg-x-div-y} and the triangle inequality, and \eqref{x-dependence-L2} holds for $K^e$ with $M = \delta$ by \eqref{reg-x-div-2} and the triangle inequality again. 

We proceed now as in the beginning of \ref{step:step2prove} in the proof of Theorem~\ref{thm-interior-linear-x}. Let $\L^e_0 \in \G_s(\lambda, \Lambda;\theta)$ be the translation invariant operator with kernel $K^e(0, y)$.
By the regularity estimates for translation invariant equations, Theorem~\ref{thm-interior-linear-2}, we have
\begin{equation}\label{ngfhhghg2222-2} 
\|u\|_{C^{\beta}(B_{1/4})} \le C\left(\|u\|_{L^\infty_{2s-\eps}(\R^n)} + \|\L^e_0u\|_{C^{\theta}(B_{1/2})}\right).
\end{equation}
Moreover, 
\begin{equation}\label{ngfhhghg2222-1} \|\L^e_0u\|_{C^\theta(B_{1/2})} \leq \|f\|_{C^\theta(B_{1/2})}+\|\L^o(u,\cdot)\|_{C^\theta(B_{1/2})} + \|\L^e(u,\cdot) - \L^e_0u\|_{C^\theta(B_{1/2})},
\end{equation}
and thanks to \eqref{ngfhhghg}  in \ref{step:step2prove} of the proof of Theorem~\ref{thm-interior-linear-x} (since $\L^e$ is now an operator in non-divergence form) we have 
\begin{equation}\label{ngfhhghg2222} 
\|\L^e(u,\cdot) - \L^e_0u\|_{C^\theta(B_{1/2})} \leq C\left(\delta\|u\|_{C^{\beta}(B_1)} + \|u\|_{L^\infty_{2s-\eps}(\R^n)}\right).
\end{equation}

It only remains to be bounded the $C^\theta$ norm of  $\L^o(u, \cdot)$. 

\item  That is, we now want to prove 
\begin{equation}\label{ngfhhghg3} 
\|\L^o(u,\cdot) \|_{C^{\theta}(B_{1/2})} \leq C\left(\bar \delta\|u\|_{C^{\beta}(B_1)} + \|u\|_{L^\infty_{2s-\eps}(\R^n)}\right),
\end{equation}
for some $\bar \delta > 0$ that is small whenever $\delta$ is small. In fact, we will show 
\begin{equation}\label{ngfhhghg3_2} 
\|\L^o(u,\cdot) \|_{C^{\theta}(B_{1/2})} \leq C\left(\delta_1 \|u\|_{C^{\beta_1}(B_1)} + \|u\|_{L^\infty_{2s-\eps}(\R^n)}\right),
\end{equation}
where
\begin{equation}\label{eq:deltabeta} 
\delta_1 := \delta^{\min\{1-\eps, \frac{2s}{\alpha}\}},\qquad \beta_1 := \min\{1, \beta\}.
\end{equation}

We will use that the operator $\L^o$ has a kernel $K^o$ that satisfies (combining upper ellipticity and \eqref{reg-x-div-even-2})
\begin{equation}\label{eq:Lo0}
\int_{B_{2\rho}\setminus B_\rho} \big|K^o(x,dy)\big| \leq C\rho^{-2s} \min\{1, \delta \rho^\alpha\}\quad\text{for all}\  \rho > 0, \, x\in \R^n, 
\end{equation}
as well as 
\begin{equation}\label{eq:Lo1}
\int_{B_{2\rho} \setminus B_\rho } \big|K^o(x+h,dy)-K^o(x,dy)\big| \leq \delta |h|^\alpha \rho^{-2s}\quad\text{for all}\  \rho > 0, \, x\in \R^n, 
\end{equation}
(by \eqref{reg-x-div-2} and the triangle inequality). 

  We start with the $L^\infty$ bound. For any $x\in B_{1/2}$ we have 
\[
\begin{split}
|\L^o(u, x)|& \le   \int_{\R^n}\left|u(x) - u(x+y)\right|\left|K^o(x, dy)\right| \\
&  \le  \|u\|_{C^{\beta_1}(B_1)} \int_{B_{1/2}}|y|^{\beta_1}\left|K^o(x, dy)\right| \\
&\quad +  C\|u\|_{L^\infty_{2s-\eps}(\R^n)} \int_{B^c_{1/2}}|y|^{2s-\eps}\left|K^o(x, dy)\right|.
\end{split}
\]
We   split each integral in dyadic balls and thanks to \eqref{eq:Lo0} and the fact that $\beta_1 + \alpha - 2s  \ge \theta > 0$ we get
\begin{equation}
\label{eq:Loinf}
|\L^o(u, x)|\le    C\left( \delta \|u\|_{C^{\beta_1}(B_1)}+\|u\|_{L^\infty_{2s-\eps}(\R^n)}\right),
\end{equation}
which gives the $L^\infty$ bound in  \eqref{ngfhhghg3_2} and \eqref{ngfhhghg3}.

The next step is to bound the $C^\theta$ seminorm. To do that, we consider $u=u_1+u_2$ where $u_1:=u\eta$ with $\eta\in C^\infty_c(B_{3/4})$, $0\le \eta \le 1$, and $\eta \equiv 1 $ in $B_{2/3}$, and bound each of the seminorms for $\L^o(u_1, x)$ and $\L^o(u_2, x)$ separately. 

\item  We focus our attention first on finding a bound for the seminorm of $\L^o(u_1, x)$, where we recall that $u_1\in C^\beta_c(B_1)$ with $\|u_1\|_{C^\beta(\R^n)}\le C \|u\|_{C^\beta(B_1)}$. Let us denote, given $\bar x\in B_{1/2}$ fixed, $\L_{\bar x}^o$ to be the translation invariant operator with kernel $K^o(\bar x, y)$ (which is not necessarily positive). We will bound, for any $x_1, x_2\in B_{1/2}$,
\begin{equation}
\label{eq:twotermsdif}
\left|\L^o(u_1, x_1) - \L^o(u_1, x_2)\right|\le |\L_{x_1}^o u_1(x_1)-\L_{x_2}^o u_1(x_1)|+|\L_{x_2}^o u_1(x_1)-\L_{x_2}^o u_1(x_2)|.
\end{equation}
For the first term, we have
\[
\begin{split}
|\L_{x_1}^o u_1(x_1)\hspace{-0.2mm}-\hspace{-0.2mm}\L_{x_2}^o u_1(x_1)|\hspace{-0.5mm}& \le\hspace{-0.5mm}\int_{\R^n} \hspace{-0.5mm}\left|u_1(x_1) - u_1(x_1+y)\right|\left|K^o(x_1, dy) \hspace{-0.5mm}- \hspace{-0.5mm}K^o(x_2, dy)\right|  \\
& = I_1 +I_2,
\end{split}
\]
where, since $u_1(x_1+y) = 0$ for $y\in \R^n\setminus B_2$, by denoting $r := |x_1-x_2|$,
\[
I_1 := |u_1(x_1)|\int_{\R^n\setminus B_2 } |K^o(x_1, dy)-K^o(x_2, dy)|  \le C \delta \|u\|_{L^\infty(B_1)}{r}^\alpha
\]
(using \eqref{eq:Lo1}), and 
\[
\begin{split}
I_2  & := \int_{B_2}\left|u_1(x_1) - u_1(x_1+y)\right|\left|K^o(x_1, dy)  -  K^o(x_2, dy)\right| \\
& \le C\|u\|_{C^{\beta_1}(B_1)} \int_{B_2}|y|^{\beta_1}\left|K^o(x_1, dy)  -  K^o(x_2, dy)\right| \\
& \le   C\|u\|_{C^{\beta_1}(B_1)} \sum_{\substack{\rho = 2^{-k}\\k \ge 0}} \rho^{\beta_1}\int_{B_{2\rho} \setminus B_{\rho}}\left|K^o(x_1, dy)  -  K^o(x_2, dy)\right|,
\end{split}
\]
(recall \eqref{eq:deltabeta}). By \eqref{eq:Lo0}-\eqref{eq:Lo1} we have  
\[
\begin{split}
I_2  & \le  C \delta\|u\|_{C^{\beta_1}(B_1)} \sum_{\substack{\rho = 2^{-k}\\k \ge 0}}  \rho^{\beta_1 - 2s} \min\{\rho^{\alpha}, {r}^\alpha\},\end{split}
\]
and we can split the sum into 
\[
\begin{split}
   {r}^\alpha \sum_{\substack{\rho = 2^{-k}\\1\ge\rho \ge {r}}}  \rho^{\beta_1 - 2s} +\sum_{\substack{\rho = 2^{-k}\\\rho \le {r}}}  \rho^{\beta_1 - 2s+\alpha},
\end{split}
\]
where the second term can be bounded by $C{r}^{\beta_1-2s+\alpha}$, since we have that ${\beta_1-2s+\alpha} > 0$;  and the first term is bounded by 
\[
 {r}^\alpha \sum_{\substack{\rho = 2^{-k}\\1\ge \rho \ge {r}}}  \rho^{\beta_1- 2s}\le \left\{
 \begin{array}{ll}
 C{r}^\alpha & \quad \text{if}\quad s < \frac12,\\
 C{r}^\alpha|\log{r}|& \quad \text{if}\quad s = \frac12,\\
  C{r}^{1+\alpha - 2s} & \quad \text{if}\quad s > \frac12.
 \end{array}
 \right.
\]
Using that ${r}^\alpha|\log{r}|\le C_\eps {r}^{\alpha -\eps}$, in all cases we have 
\[
I_2 \le C \delta \|u\|_{C^{\beta_1}(B_1)}  {r}^{\theta}.
\]
Together with the bound on $I_1$ and the fact that $\alpha \ge \theta$, we obtain 
\begin{equation}
\label{eq:prevtog}
|\L_{x_1}^o u_1(x_1)\hspace{-0.2mm}-\hspace{-0.2mm}\L_{x_2}^o u_1(x_1)| \le C\delta \|u\|_{C^{\beta_1}(B_1)} {r}^\theta. 
\end{equation}

Now, for the second term in \eqref{eq:twotermsdif} we use that, since $\beta_1 \le 1$, 
\[
|u_1(x_1) - u_1(x_1+y) - u_1(x_2) + u_1(x_2+y)|\le C\|u\|_{C^{\beta_1}(B_1)}\min\{{r}^{\beta_1}, |y|^{\beta_1}\},
\]
and thus, by \eqref{eq:Lo0}, 
\[
|\L_{x_2}^o u_1(x_1)-\L_{x_2}^o u_1(x_2)| \le C \|u\|_{C^{\beta_1}(B_1)}\sum_{\rho = 2^k} \rho^{-2s} \min\{\rho^{{\beta_1}} , {r}^{{\beta_1}}\} \min\{1, \delta \rho^\alpha\}.
\]
We split the sum into three terms according to the value of $\rho\in (0, \infty) = (0, r)\cup (r,   \delta^{-\frac{1}{\alpha}}) \cup (\delta^{-\frac{1}{\alpha}}, \infty)$ and bound it by
\begin{equation}
\label{eq:three_terms}
\delta \sum_{\substack{\rho = 2^k\\\rho \le  r}}  \rho^{-2s+\alpha+{\beta_1}} +\delta{r}^{\beta_1} \sum_{\substack{\rho = 2^k\\r \le \rho \le \delta^{-\frac{1}{\alpha}}}} \rho^{-2s+\alpha} +\sum_{\substack{\rho = 2^k\\ \delta^{-\frac{1}{\alpha}}\le\rho }} \rho^{-2s}{r}^{\beta_1}.
\end{equation}
The first term is immediately bounded by $C\delta r^{-2s+\alpha+{\beta_1}}$, since $-2s+\alpha +{\beta_1} > 0$; and the third term is bounded by $C{r}^{\beta_1} \delta^{\frac{2s}{\alpha}}$. Observe that, since $-2s+\alpha+\beta_1\ge\beta - 2s$ and $\beta_1\ge \beta - 2s$, we have that the first and third terms are bounded by $C\delta_1 r^{\beta-2s}$ (recall \eqref{eq:deltabeta}).

 For the second term, we have different values according to the relative values of $\alpha$ and $s$, as follows:
\begin{itemize}
\item If $2s < \alpha$, then the second term in \eqref{eq:three_terms} is bounded by  $C\delta{r}^{\beta_1} \delta^{\frac{2s-\alpha}{\alpha}} = C{r}^{\beta_1} \delta^{\frac{2s}{\alpha}}$. 
\item If $2s > \alpha$, then the second term in \eqref{eq:three_terms} is bounded by  $C\delta r^{-2s+\alpha+{\beta_1}}$. 
\item Finally, in the case $2s = \alpha\le 1$, the second term is bounded by the factor $C\delta {r}^{\beta_1}\left(|\log r|+|\log \delta|\right)\le C_{ \eps} \delta^{1- \eps}{r}^{{\beta_1}- \eps}$. We have $\beta_1-\eps \ge \beta-2s$ as well (since $\eps$ was small).
\end{itemize}

Putting all terms together, we have shown that the sum \eqref{eq:three_terms} is bounded by $\delta_1 r^{\beta-2s}$ (recall \eqref{eq:deltabeta}) and thus we have 
\[
|\L_{x_2}^o u_1(x_1)-\L_{x_2}^o u_1(x_2)| \le C \delta_1 \|u\|_{C^{\beta_1}(B_1)}
 r^{\theta}.
\]
With \eqref{eq:prevtog}, this gives in \eqref{eq:twotermsdif}
\begin{equation}
\label{eq:twotermsdif2}
\left|\L^o(u_1, x_1) - \L^o(u_1, x_2)\right|\le  C \delta_1 \|u\|_{C^{\beta_1}(B_1)}|x_1-x_2|^{\theta},
\end{equation}
which bounds the seminorm $[\L^o(u_1, \cdot)]_{C^{\theta}(B_{1/2})}\le C\delta_1 \|u\|_{C^{\beta_1}(B_1)}$.

 \item Let us now bound the $C^\theta$ seminorm of $\L^o(u_2, x)$, where now $u_2$ satisfies that $u_2 \equiv 0$ in $B_{2/3}$ and $\|u_2\|_{L^\infty_{2s-\eps}(\R^n)}\le C\|u\|_{L^\infty_{2s-\eps}(\R^n)}$. To do it, we will use the regularity of the kernel in the $y$ variable, namely, \eqref{reg-x-div-y}. By triangle inequality and in terms of $K^o$, this condition reads as 
 \begin{equation}\label{reg-x-div-y_o}
\int_{B_{2\rho} \setminus B_\rho} \left|K^o(x, h+dy) - K^o(x, dy)\right|  \le M |h|^\theta\rho^{-2s-\theta} \quad\text{for all}\quad h\in B_{\rho/2},
\end{equation}
for all $x\in \R^n$, $\rho > 0$.

 Using the same notation as before, we again split as, for any $x_1, x_2\in B_{1/2}$,
\begin{equation}
\label{eq:twotermsdif_222}
\left|\L^o(u_2, x_1) - \L^o(u_2, x_2)\right|\le |\L_{x_1}^o u_2(x_1)-\L_{x_2}^o u_2(x_1)|+|\L_{x_2}^o u_2(x_1)-\L_{x_2}^o u_2(x_2)|.
\end{equation}
 Now, for the first  term we have, since $u_2(x_1) = u_2(x_2) = 0$, 
\[
\begin{split}
|\L_{x_1}^o u_2(x_1)-\L_{x_2}^o u_2(x_1)|&  \le \int_{\R^n\setminus B_{1/6}} |u_2(x_1+y)| \left|K^o(x_1, dy) - K^o(x_2, dy)\right| \\
& \le C \delta   \|u\|_{L^\infty_{2s-\eps}(\R^n)} |x_1-x_2|^\alpha, 
\end{split}
\]
where we have used the fact that $|u_2(z)|\le C(1+|z|^{2s-\eps}) \|u\|_{L^\infty_{2s-\eps}(\R^n)}$ together with \eqref{eq:Lo1}. 

For the second term, we have 
\[
\begin{split}
|\L_{x_2}^o u_2(x_1)\hspace{-0.3mm}-\hspace{-0.3mm}\L_{x_2}^o u_2(x_2)| & \hspace{-0.3mm}\le\hspace{-0.3mm} \int_{ B^c_{1/2}}\hspace{-1.5mm} |u_2(z)|\left|K^o(x_2, dz\hspace{-0.3mm}-\hspace{-0.3mm}x_1) \hspace{-0.3mm}-\hspace{-0.3mm} K^o(x_2, dz\hspace{-0.3mm}-\hspace{-0.3mm}x_2)\right|  \\
& \le C M \|u\|_{L^\infty_{2s-\eps}(\R^n)} |x_1-x_2|^\theta, 
\end{split}
\]
where we have used again the bound on $u_2$ now together with  \eqref{reg-x-div-y_o}. Since $\alpha \ge \theta$, we have shown that 
\begin{equation}
\label{eq:tihrdto}
[\L^o(u_2, \cdot)]_{C^\theta(B_{1/2})}\le C (M+\delta) \|u\|_{L^\infty_{2s-\eps}(\R^n)}. 
\end{equation}
\item Thanks to \eqref{eq:Loinf}-\eqref{eq:twotermsdif2}-\eqref{eq:tihrdto} we have now shown \eqref{ngfhhghg3_2}, and thus, \eqref{ngfhhghg3}. Together with \eqref{ngfhhghg2222-2}-\eqref{ngfhhghg2222-1}-\eqref{ngfhhghg2222} this shows  
\[
\|u\|_{C^{\beta}(B_{1/4})} \le C\left(\bar \delta\|u\|_{C^{\beta}(B_1)} + \|u\|_{L^\infty_{2s-\eps}(\R^n)} + \|f\|_{C^\theta(B_{1/2})}\right)
\]
for all $u\in C^{\beta}(B_1)\cap L^\infty_{2s-\eps}(\R^n)$, where $\L(u, \cdot) = f$, and where we can still choose $\bar \delta$ small (depending only on $n$, $s$, $\alpha$, $\eps$, $\lambda$, $\Lambda$, and $M$).  The proof now follows  in a standard way (cf. \ref{step3} in the proof of Theorem \ref{thm-interior-linear-x}), after choosing $\delta$ small enough, to deduce that (after a covering argument)
\[
\|u\|_{C^{\beta}(B_{1/2})} \le C\left(\|u\|_{L^\infty_{2s-\eps}(\R^n)} + \|f\|_{C^\theta(B_{1})}\right),
\]
for all $u\in C^{\beta}(B_1)\cap L^\infty_{2s-\eps}(\R^n)$, where $\L(u, \cdot) = f$, as wanted. 
\qedhere
\end{steps}
\end{proof}

In case $\beta<2s$ (i.e., $s>\frac12$ and $\beta=1+\alpha<2s$) the equation cannot be seen as a non-divergence-form equation, and thus we need a different argument.
We will proceed by a compactness argument, by first showing a quantitative Liouville-type estimate for solutions in very large balls.
 
\begin{prop}
\label{prop:compactness-x-div}
Let $s\in (\frac12,1)$, $\delta > 0$, and $\alpha \in(0,1)$ be  such that $1+\alpha<2s$.
Let $\L$ be an operator of the form \eqref{divergence-form0}-\eqref{divergence-form1}, with kernels satisfying the ellipticity conditions \eqref{divergence-ellipticity0}-\eqref{divergence-ellipticity1}.
Let $u\in C^{1+\alpha}(\R^n)\cap L^\infty_{2s-\eps}(\R^n)$ for some $\eps>0$ and $f\in L^q_{\rm loc}(\R^n)$ for some $q \ge 1$, and with $[u]_{C^{1+\alpha}(\R^n)}\le 1$ and $\|f\|_{L^q(B_{{1}/{\delta}})} \le \delta$.
Assume in addition that 
\begin{equation}\label{x-dependence-div}
\begin{split} 
& \|\nabla u\|_{L^\infty(\R^n)} \hspace{-1mm}\int_{B_{2\rho}(x)\setminus B_\rho(x)} \hspace{-1mm}\big|K(x\hspace{-0.5mm}+\hspace{-0.5mm}h,dz\hspace{-0.5mm}+\hspace{-0.5mm}h)\hspace{-0.5mm}-\hspace{-0.5mm}K(x,dz)\big|    \leq \delta |h|^\alpha \rho^{-2s}  \\
 & \|\nabla u\|_{L^\infty(\R^n)} \hspace{-1mm}\int_{B_{2\rho}\setminus B_\rho}\hspace{-1mm} \big|K(x,x+dy)\hspace{-0.5mm}-\hspace{-0.5mm}K(x, x-dy)\big|    \leq \delta \rho^{\alpha-2s}
\end{split}
\end{equation}
for all $x,h\in \R^n$, and $\rho > 0$.
Suppose also that $u$ satisfies
\[
\L(u,x) = f\quad\text{in}\quad B_{{1}/{\delta}}
\]
in the weak sense.

Then, for every $\eps_\circ > 0$ there exists $\delta_\circ >0 $ depending only on $\eps_\circ$, $n$, $s$, $\alpha$, $q$, $\eps$, $\lambda$, and $\Lambda$,  such that if $\delta < \delta_\circ$,
\[
\|u-\ell\|_{C^{1}(B_1)}\le \eps_\circ,
\]
where $\ell(x)=u(0)+\nabla u(0)\cdot x$.
\end{prop}

\begin{proof}
  For the sake of readability, we will assume here that 
\[
K(x, dz) = K(x, z)\ dz,
\]
so that, in particular, $K(x, z) = K(z, x)$ from \eqref{divergence-form1}.  
We divide it into three steps:
\begin{steps}
\item Assume that the statement does not hold. 
Then, there exists some $\eps_\circ > 0$ such that for any $k\in \N$, there are $u_k\in C_{\rm loc}^{1+\alpha}(\R^n)\cap L^\infty_{2s-\eps}(\R^n)$ with $[u_k]_{C^{1+\alpha}(\R^n)}\le 1$, $f_k\in L^q(B_k)$ with $\|f_k\|_{L^q(B_k)}\le \frac{1}{k}$, and $\L^{(k)}$ as in the statement such that
\[
\L^{(k)} (u_k, x) = f_k \quad\text{in}\quad B_k,
\]
in the weak sense, with
\begin{equation}
\label{eq:byassumptdiv}
\begin{split} 
&\|\nabla u_k\|_{L^\infty(\R^n)}\int_{B_{2\rho}(x)\setminus B_\rho(x)} \big|K^{(k)}(x+\bar h,z+\bar h)-K^{(k)}(x,z)\big|dz \leq \frac{1}{k} |\bar h|^\alpha \rho^{-2s}\\
& \|\nabla u_k\|_{L^\infty(\R^n)}\int_{B_{2\rho}\setminus B_\rho} \big|K^{(k)}(x,x+y)-K^{(k)}(x, x-y)\big|dy \leq \frac{1}{k} \rho^{\alpha-2s},
\end{split}
\end{equation}
for all $x, \bar h\in \R^n$, and $\rho > 0$,  but 
\[
\|u_k - \ell_k\|_{C^1(B_1)}\ge \eps_\circ.
\]

Let us define, for a fixed $h\in \R^n$,
\[
 V_k := u_k(x+h)+u_k(x-h)-2u_k(x), \quad \ F_k := f_k(x+h)+f_k(x-h)-2f_k(x),
\]
with $\|F_k\|_{L^q(B_k)} \le \frac{4}{k}$, $\|V_k\|_{C^{1+\alpha}(\R^n)} \leq C_h$.

\item Let $K_0^{(k)}(y)$ be the kernel denoting the even part of $K(x, x+y)$ at $x = 0$, i.e., $K_0^{(k)}(y) := \frac12K^{(k)}(0,y)+\frac12K^{(k)}(0,-y)$. Observe that by assumption, the operators with kernel $K_0^{(k)}$ belong to $\GL$. 

Now, for any $\eta\in C^\infty_c(\R^n)$ we have (for $k$ such that ${\rm supp}\,\eta\subset B_k$)
\[\int_{\R^n} \int_{\R^n} \hspace{-2mm} \big(V_k(x)-V_k(z)\big)\big(\eta(x)-\eta(z)\big) K^{(k)}_0(z-x)dz\, dx = 2\int_{\R^n}\hspace{-2mm}  F_k\eta+ 2E_0-E_h-E_{-h},\]
where
\[E_0:= \int_{\R^n} \int_{\R^n}  E_0(x, z)\, dz\,dx,\]
with 
\[
E_0(x, z) := \big(u_k(x)-u_k(z)\big)\big(\eta(x)-\eta(z)\big) \big(K^{(k)}(x,z) - K^{(k)}_0(z-x)\big),
\]
and the expressions for $E_{\pm h}$ are analogous, replacing $u_k(x)$, $u_k(z)$, and $K^{(k)}(x,z)$ by $u_k(x\pm h)$, $u_k(z\pm h)$,  and $K^{(k)}(x\pm h,z\pm h)$ respectively. By symmetry in the roles of $x$ and $z$ (here we use that $K(x, z) = K(z, x)$, or \eqref{divergence-form1} in the case of non-absolutely continuous measures),
\begin{equation}
\label{eq:absE0}
\begin{split}
|E_0|& \le \int_{{\rm supp}\, \eta}\int_{\R^n} |E_0
(x, z)|\, dz\, dx+\int_{\R^n}\int_{{\rm supp}\, \eta} |E_0 (x, z)|\, dz\, dx\\
& \le 2\int_{{\rm supp}\, \eta}\int_{\R^n} |E_0
(x, z)|\, dz\, dx.
\end{split}
\end{equation}

Using that $u_k$ are globally Lipschitz and that $\eta\in C^1(\R^n)$ we have
\[
\begin{split} \big|u_k(x)-u_k(z)\big|\,\big|\eta(x)-\eta(z)\big| & \leq  C\|\nabla u\|_{L^\infty(\R^n)}|x-z|\min\{1, |x-z|\},
\end{split}\]
for some constant depending only on $\eta$. Hence, from \eqref{eq:absE0} we have
\begin{equation}
\label{eq:E0bound}
|E_0|\le C \sum_{\substack{\rho = 2^j\\ j \in \Z}}  \rho \min\{1, \rho\}\int_{{\rm supp}\, \eta}I_\rho(x) \, dx.
\end{equation}
where
\[
I_\rho(x) := \|\nabla u\|_{L^\infty(\R^n)}\int_{B_{2\rho}\setminus B_\rho}
\left| K^{(k)}(x, x+y)-K_0^{(k)}(y)\right|\, dy.
\]

We split
\[
\begin{split}
\left|K^{(k)}(x, x+y)-K_0^{(k)}(y)\right|  & \le \frac12 \left|K^{(k)}(x, x+y)-K^{(k)}(0, y)\right| \\
& \quad +\frac12 \left|K^{(k)}(x, x-y)-K^{(k)}(0, -y)\right|\\
& \quad +\frac12 \left|K^{(k)}(x, x+y)-K^{(k)}(x, x-y)\right|,
\end{split}
\] 
so that 
\[
\begin{split}
I_\rho (x)&  \le \|\nabla u\|_{L^\infty(\R^n)}\int_{B_{2\rho}\setminus B_\rho}
\left| K^{(k)}(x, x+y)-K^{(k)}(0, y)\right|\, dy\\
& \quad +\frac{\|\nabla u\|_{L^\infty(\R^n)}}{2} \int_{B_{2\rho}\setminus B_\rho}
\left| K^{(k)}(x, x+y)-K^{(k)}(x, x-y)\right|\, dy.
\end{split}
\]
The first term can be bounded thanks to the first inequality in \eqref{eq:byassumptdiv} (putting $\bar h = x$ and $x = 0$), and the second term is directly bounded thanks to the second inequality in \eqref{eq:byassumptdiv}, so that we obtain 
\[
I_\rho(x) \le \frac{1}{k} \rho^{-2s}\left(|x|^\alpha +   \rho^\alpha\right). 
\]
Putting it back into the bound on $|E_0|$, \eqref{eq:E0bound}, we get
\[
 |E_0|\le \frac{C}{k } \sum_{\substack{\rho = 2^j\\ j \in \Z}}  \rho^{1-2s} \min\{1, \rho\}\int_{{\rm supp}\, \eta}\left(|x|^\alpha +   \rho^\alpha\right)\, dx.
\]
Since $\eta$ is fixed and compactly supported, the last integral is finite and bounded  by $(1+\rho^\alpha)$ (up to a constant depending only on $\eta$), so that 
\[
 |E_0|\le \frac{C}{k } \sum_{\substack{\rho = 2^j\\ j \in \Z}}  \rho^{1-2s} \min\{1, \rho\}\left(1+   \rho^\alpha\right) \le \frac{C}{k},
\]
where the last sum is finite since 
\[
\rho^{1-2s} \min\{1, \rho\}\left(1+   \rho^\alpha\right)\le\left\{
\begin{array}{ll}
2\rho^{1-2s+\alpha}&\quad\text{if} \ \rho \ge 1,\\
2\rho^{2-2s}&\quad\text{if} \ \rho < 1,
\end{array}
\right.
\]
and  $1 +\alpha < 2s < 2$.

\item We have proved $|E_0|\le \frac{C}{k}$, and the same bounds hold for $E_{\pm h}$ as well (for example, simply by considering the test functions $\eta(\cdot\pm h)$ instead of $\eta$). 
Thus, together with the fact that  $\|F_k\|_{L^q(B_k)} \le \frac{4}{k}$, we get 
\[\int_{\R^n} \int_{\R^n} \big(V_k(x)-V_k(z)\big)\big(\eta(x)-\eta(z)\big) K^{(k)}_0(z-x)dz\,dx \longrightarrow 0 \quad \textrm{as}\ k\to\infty.\]
By  Arzel\`a-Ascoli, the functions $V_k$ converge (up to a subsequence) in $C^1_{\rm loc}(\R^n)$ to a function $V\in C^{1+\alpha}(\R^n)$. 
On the other hand,   the measures $\min\{1,|y|^2\}K^{(k)}_0(dy)$ converge weakly to a limiting measure $\min\{1,|y|^2\}\bar K_0(dy)$ that will satisfy 
\[\int_{\R^n} \int_{\R^n} \big(V(x)-V(z)\big)\big(\eta(x)-\eta(z)\big) \bar K_0(z-x)dz\, dx= 0.\]
Notice  that since $V\in C^{1+\alpha}(\R^n)$, it has finite energy  on compact sets. Together with the fact that the previous equality holds  for any $\eta\in C^\infty_c(\R^n)$, we have that $V$ solves $\bar \L_0 V=0$ in $\R^n$ in the weak sense (where $\bar \L_0\in \GL$ is the limiting operator with kernel $\bar K_0$), and by Liouville's theorem, Theorem~\ref{thm:Liouville}, we get that $V$ is constant.
If we define 
\[
v_k := u_k - \ell_k,
\]
then, $v_k(0) = |\nabla v_k(0)|  = 0$ with $[v_k]_{C^{1+\alpha}(\R^n)}\le 1$ and $v_k\to v$ in $C^1_{\rm loc}$, for some $v$ with $v(0) = |\nabla v(0)| = 0$ and $[v]_{C^{1+\alpha}(\R^n)}\le 1$. Since  $V_k(x) = v_k(x+h) + v_k(x-h) - 2 v_k(x)$, we get $V(x)=v(x+h)+v(x-h)-2v(x)$, which is constant (for every $h\in \R^n$ fixed). By Lemma~\ref{it:H10} we have that $v$ is a quadratic polynomial, and the condition $[v]_{C^{1+\alpha}(\R^n)}\le 1$ implies it is actually linear. Because it also satisfies $v(0) = |\nabla v(0)| = 0$, it must be $v \equiv 0$, which is a contradiction with $\|v\|_{C^1(B_1)}\ge \eps_\circ>0$. 
\qedhere
\end{steps} 
\end{proof}

Thanks to the previous Liouville-type statement, we get the following estimate, which is almost the desired result in case $\beta < 2s$:

\begin{prop}
\label{prop:interior-linear-x-div}
Let $s\in (\frac12,1)$ and $\alpha\in (0,1)$ be such that $1+\alpha<2s$, and let $q = \frac{n}{2s-1-\alpha}$. Let $\L$ be an operator of the form \eqref{divergence-form0}-\eqref{divergence-form1}, with kernels satisfying \eqref{divergence-ellipticity0}-\eqref{divergence-ellipticity1} and \eqref{reg-x-div}-\eqref{reg-x-div-even} for some $M > 0$.
Then, the following holds. 

For any $\delta > 0$ there exists $C_\delta$ such that
\[
[u]_{C^{1+\alpha}(B_{1/2})} \le \delta [u]_{C^{1+\alpha}(\R^n)} + C_\delta \left(\|u\|_{L^\infty(B_1)} +\|\nabla u\|_{L^\infty(\R^n)} + \|f\|_{L^q(B_1)}\right)
\]
for any $u \in C^{1+\alpha}_c(\R^n)$ satisfying $\L(u,x)=f$ in $B_1$ in the weak sense.
The constant $C_\delta$ depends only on $\delta$, $n$, $s$, $\alpha$,    $M$, $\lambda$, and $\Lambda$.
\end{prop}

\begin{proof}
Let us denote,  for $w\in C_c^{1+\alpha}(\R^n)$,
\[
\tilde{\mathcal{S}} (w) := \inf\left\{\|g\|_{L^q(B_1)} :  
\begin{array}{l}
\tilde \L(w,x)=g \ \text{in the weak sense, for some $\tilde\L$ of}\\
\text{the form \eqref{divergence-form0}-\eqref{divergence-form1}-\eqref{divergence-ellipticity0}-\eqref{divergence-ellipticity1}}\\\text{and satisfying \eqref{reg-x-div}-\eqref{reg-x-div-even}, with $M>0$.}
\end{array}\right\}.
\]

We use Lemma~\ref{lem-interior-blowup} with $\mu = 1+\alpha$ and
\[
\mathcal{S}(w) =  \left\{
\begin{array}{ll} \displaystyle
\tilde{\mathcal{S}} (w) + \|\nabla w\|_{L^\infty(\R^n)} & \text{if}\quad w\in  C_c^{1+\alpha}(\R^n).
\\
\infty & \text{otherwise,}
\end{array}
\right.
\]
Notice that the mapping $\mathcal{S}:C^{1+\alpha}(\R^n)\to \R_{\ge 0}$ depends only on $n$, $s$, $\alpha$,    $M$, $\lambda$, and~$\Lambda$.

Thus, either Lemma~\ref{lem-interior-blowup}~\ref{it:lem_int_blowup_i} holds, in which case we would have
\[
[u]_{C^{1+\alpha}(B_{1/2})} \le \delta [u]_{C^{1+\alpha}(\R^n)} + C_\delta \left(\|u\|_{L^\infty(B_1)} +\|\nabla u\|_{L^\infty(\R^n)} + \|f\|_{L^q(B_1)}\right),
\]
or there exists a sequence $u_k\in C^{1+\alpha}_c(\R^n)$ and $\L_{k}$ of the previous form such that $\L_k(u_k,x)=f_k$ in the weak sense,
\begin{equation}
\label{eq:fktozero-x-div}
\frac{\|f_k\|_{L^q(B_1)}+\|\nabla u_k\|_{L^\infty(\R^n)}}{[u_k]_{C^{1+\alpha}(B_{1/2})}} \le \frac{2\mathcal{S}(u_k)}{[u_k]_{C^{1+\alpha}(B_{1/2})}}\to 0,
\end{equation}
and for some $x_k \in B_{1/2}$ and $r_k \downarrow 0$,
\[
v_k(x) := \frac{u_k(x_k+r_k x)}{r_k^{1+\alpha}[u_k]_{C^{1+\alpha}(\R^n)}}
\]
satisfies 
\begin{equation}
\label{eq:vkcontradiction-x-div}
\|v_k - \ell_k\|_{C^1(B_1)} > \frac{\delta}{2},
\end{equation}
where $\ell_k$ is the 1st order Taylor polynomial of $v_k$ at 0.  
Then, by scaling, there exists an operator of the form \eqref{divergence-form0}-\eqref{divergence-form1}-\eqref{divergence-ellipticity0}-\eqref{divergence-ellipticity1}, $\tilde \L_k$, such that 
\[
\tilde \L_k(v_k,x) = r_k^{2s-1-\alpha}\frac{f_k(x_k+r_k x)}{ [u_k]_{C^{1+\alpha}(\R^n)}} =:\tilde f_k(x)
\]
in the weak sense.
More precisely, if $\L_k$ has kernel $K^k(x,z)$, then $\tilde \L_k$ has kernel $\tilde K^k(x,dz)$ given by 
\[
\tilde K^k (x,dz) := r_k^{2s} K^k (x_k+r_kx,x_k+r_k \, dz).
\]
Moreover, since $\L_k$ satisfies \eqref{reg-x-div}-\eqref{reg-x-div-even}, $\tilde\L_k$ also satisfies them with a smaller $M$, i.e., as in \ref{step:1div} of the proof of Theorem~\ref{thm-interior-linear-x-div} in case $\beta>2s$ (on page \pageref{step:1div}) we have 
\[
\int_{B_{2\rho}(x)\setminus B_\rho(x)} \big|\tilde K^k(x+h,h+dz)-\tilde K^k(x,dz)\big| \leq   Mr_k^{\alpha} |h|^\alpha \rho^{-2s},
\]
and 
\[
\int_{B_{2\rho}\setminus B_\rho} \big|\tilde K^k(x,x+dy)-\tilde K^k(x,x-dy)\big|  \leq Mr_k^\alpha \rho^{\alpha-2s}
\]
for all $x, h\in \R^n$,    and $\rho > 0$.
We also have
\[   \|\nabla v_k\|_{L^\infty(\R^n)} \leq \frac{r_k^{-\alpha}\|\nabla u_k\|_{L^\infty(\R^n)}}{[u_k]_{C^{1+\alpha}(\R^n)}},\]
so that 
\[
\begin{split}
&  \|\nabla v_k\|_{L^\infty(\R^n)} \hspace{-1mm}\int_{B_{2\rho}(x)\setminus B_\rho(x)}\hspace{-1mm} \big|\tilde K^k(x+h,h+dz)-\tilde K^k(x,dz)\big| \leq  \\
& \hspace{7.5cm}\le  \frac{M \|\nabla u_k\|_{L^\infty(\R^n)}}{[u_k]_{C^{1+\alpha}(\R^n)}} |h|^\alpha \rho^{-2s},
\\
&  \|\nabla v_k\|_{L^\infty(\R^n)}\hspace{-1mm} \int_{B_{2\rho}\setminus B_\rho}\hspace{-1mm} \big|\tilde K^k(x,x+dy)-\tilde K^k(x,x-dy)\big|  \leq \frac{M \|\nabla u_k\|_{L^\infty(\R^n)}}{[u_k]_{C^{1+\alpha}(\R^n)} \rho^{2s-\alpha}}
\end{split}
\]
for all $x, h\in \R^n$, $\rho > 0$, and with $\|\nabla u_k\|_{L^\infty(\R^n)} / [u_k]_{C^{1+\alpha}(\R^n)} \to 0$ as $k\to \infty$ (thanks to \eqref{eq:fktozero-x-div}).

Also from \eqref{eq:fktozero-x-div} and since $q = \frac{n}{2s-1-\alpha}$,
\[
\|\tilde f_k\|_{L^q (B_{1 /(2r_k)})} 
= \frac{\|f_k(x_k+r_k\,\cdot\,)\|_{L^q(B_{1 /(2r_k)})}}{r_k^{1+\alpha-2s} [u_k]_{C^{1+\alpha}(\R^n)}}
\leq  r_k^{2s-1-\alpha-\frac{n}{q}}\frac{\|f_k\|_{L^q(B_1)}}{[u_k]_{C^{1+\alpha}(\R^n)}}
\to 0, 
\] 
as $k\to \infty$. 
In all, since by definition $[v_k]_{C^{1+\alpha}(\R^n)} = 1$,  we have that $v_k$ satisfies all the hypotheses of Proposition~\ref{prop:compactness-x-div} for any fixed $\delta_\circ>0$, if $k$ is large enough. 
In particular, taking $\eps_\circ$ sufficiently small in Proposition~\ref{prop:compactness-x-div} we get a contradiction with \eqref{eq:vkcontradiction-x-div}.
\end{proof}

We can now give the final part of the proof of Theorem \ref{thm-interior-linear-x-div}:

\begin{proof}[Proof of Theorem \ref{thm-interior-linear-x-div} (and \ref{thmC}) in case $\beta<2s$]
Notice that, since $\beta<2s$, we have $\beta=1+\alpha<2s$ and $s\in(\frac12,1)$.
By Proposition~\ref{prop:interior-linear-x-div}, for any $\delta> 0$ there exists $C_\delta$ depending only on $\delta$, $n$, $s$, $\alpha$,    $M$, $\lambda$, and $\Lambda$,    such that 
\begin{equation}
\label{eq:touse-x-div}
[u]_{C^{\beta}(B_{1/2})}\le \delta [u]_{C^{\beta}(\R^n)} + C_\delta \left(\|u\|_{L^\infty(B_1)}+\|\nabla u\|_{L^\infty(\R^n)}+\|f\|_{L^q(B_1)}\right)
\end{equation}
for any $u\in C^{\beta}_c(\R^n)$ satisfying $\L(u,x)=f$ in the weak sense, where $ q := \frac{n}{2s-\beta} = \frac{n}{2s-1-\alpha}$.  

Let $\eta\in C^\infty_c(B_3)$ such that $\eta \equiv 1 $ in $B_{2}$, and consider the function $u\eta$ for $u\in C^{\beta}(\R^n)\cap L^\infty_{2s-\eps}(\R^n)$ satisfying $\L(u,x)=f$ in the weak sense.
Since $u - \eta u \equiv 0$ in $B_{2}$, we have that
\begin{equation}\label{claim-x-div}
\|\L (u-\eta u,x)\|_{L^\infty(B_1)} \le C \|u\|_{L^\infty_{2s-\eps}(\R^n)}.
\end{equation}

Hence, we have $\L (\eta u,x)=g$ in the weak sense, with 
\[
\|g\|_{L^q(B_1)} \le \|f\|_{L^q(B_1)} + C \|u\|_{L^\infty_{2s-\eps}(\R^n)}.
\]
Apply now \eqref{eq:touse-x-div} to $u\eta$, to get
\begin{equation}
\label{eq:touse2-x}
[u]_{C^{\beta}(B_{1/2})}\le   \delta [u]_{C^{\beta}(B_4)} + C_\delta \left(\|u\|_{L^\infty_{2s-\eps}(\R^n)}+\|\nabla u\|_{L^\infty(B_4)}+\|f\|_{L^q(B_1)}\right)
\end{equation}
for any $u\in C^{\beta}(B_4)\cap L^\infty_{2s-\eps}(\R^n)$. By interpolation  we know that 
\[
\|\nabla u\|_{L^\infty(B_4)}\le \delta[u]_{C^\beta(B_4)} + C_\delta \|u\|_{L^\infty(B_4)}, 
\]
so that \eqref{eq:touse2-x} becomes
\begin{equation}
\label{eq:touse2-x2}
[u]_{C^{\beta}(B_{1/2})}\le   2\delta [u]_{C^{\beta}(B_4)} + C_\delta \left(\|u\|_{L^\infty_{2s-\eps}(\R^n)}+\|f\|_{L^q(B_1)}\right)
\end{equation}
for any $u\in C^{\beta}(B_4)\cap L^\infty_{2s-\eps}(\R^n)$.

Now, by a standard interpolation and covering-type argument, we deduce from \eqref{eq:touse2-x2} and for some $\delta$ small enough that
\[
\|u\|_{C^{\beta}(B_{1/2})} \le C \left(\|u\|_{L^\infty_{2s-\eps}(\R^n)}+ \|f\|_{L^q(B_1)}\right)
\]
for all $u\in C^{\beta}(B_1)\cap L^\infty_{2s-\eps}(\R^n)$ such that $\L(u, x) = f$ in $B_1$ in the weak sense, which proves Theorem \ref{thm-interior-linear-x-div}.  
\end{proof}

\section{Cordes-Nirenberg estimates in divergence form}
\label{sec5}

We now consider divergence form equations as before, but where the regularity of the operators is reduced to \emph{small oscillations}. Namely, we consider the conditions (cf.  \eqref{reg-x-div-CN}-\eqref{reg-x-div-even-CN})
\begin{equation}\label{reg-x-div-CN-m}
\int_{B_{2\rho}(x)\setminus B_\rho(x)} \big|K(x+h,h+dz)-K(x,dz)\big| \leq \eta \rho^{-2s}
\end{equation}
and
\begin{equation}\label{reg-x-div-even-CN-m}
\int_{B_{2\rho}\setminus B_\rho} \big|K(x,x+dy)-K(x,x-dy)\big|   \leq \eta\rho^{-2s}
\end{equation}
for all $x,h\in \R^n$, and $\rho > 0$, where the parameter $\eta$ will be small.

\begin{prop}
\label{prop:compactness-x-div-cn}
Let $s\in (0,1)$, $\delta > 0$, and $\beta\in (s, \min\{1, 2s\})$.
Let $\L$ be an operator of the form \eqref{divergence-form0}-\eqref{divergence-form1}, with kernels satisfying the ellipticity conditions \eqref{divergence-ellipticity0}-\eqref{divergence-ellipticity1}, and \eqref{reg-x-div-CN-m}-\eqref{reg-x-div-even-CN-m} for some $\eta\le\delta$.

Let $u\in C^{\beta}(\R^n)\cap L^\infty_{2s-\eps}(\R^n)$ for some $\eps>0$ and $f\in L^q_{\rm loc}(\R^n)$ for some $q \ge 1$, and with $[u]_{C^{\beta}(\R^n)}\le 1$ and $\|f\|_{L^q(B_{{1}/{\delta}})} \le \delta$.

Suppose also that $u$ satisfies
\[
\L(u,x) = f\quad\text{in}\quad B_{{1}/{\delta}}
\]
in the weak sense.

Then, for every $\eps_\circ > 0$ there exists $\delta_\circ >0 $ depending only on $\eps_\circ$, $n$, $s$, $\beta$, $q$, $\eps$, $\lambda$, and $\Lambda$,  such that if $\delta < \delta_\circ$,
\[
\|u-u(0)\|_{L^\infty(B_1)}\le \eps_\circ.
\]
\end{prop}
\begin{proof}
  The proof follows the lines of the proof of Proposition~\ref{prop:compactness-x-div}. We also assume here that $
K(x, dz) = K(x, z)\ dz
$. We divide it into three steps:
\begin{steps}
\item Assume that the statement does not hold. 
Then, there exists some $\eps_\circ > 0$ such that for any $k\in \N$, there are $u_k\in C_{\rm loc}^{\beta}(\R^n)\cap L^\infty_{2s-\eps}(\R^n)$ with $[u_k]_{C^{\beta}(\R^n)}\le 1$, $f_k\in L^q(B_k)$ with $\|f_k\|_{L^q(B_k)}\le \frac{1}{k}$, and $\L^{(k)}$ as in the statement such that
\[
\L^{(k)} (u_k, x) = f_k \quad\text{in}\quad B_k,
\]
in the weak sense, with 
\begin{equation}
\label{eq:byassumptdiv-cn}
\begin{split} 
& \int_{B_{2\rho}(x)\setminus B_\rho(x)} \big|K^{(k)}(x+  \bar h,z+  \bar h)-K^{(k)}(x,z)\big|dz \leq \frac{1}{k}  \rho^{-2s}\\
& \int_{B_{2\rho}\setminus B_\rho} \big|K^{(k)}(x,x+y)-K^{(k)}(x, x-y)\big|dy \leq \frac{1}{k} \rho^{ -2s},
\end{split}
\end{equation}
for all $x, \bar h\in \R^n$, and $\rho > 0$,  but 
\[
\|u_k - u_k(0)\|_{L^\infty(B_1)}\ge \eps_\circ.
\]

Let us define, for a fixed $h\in \R^n$,
\[
 V_k := u_k(x+h)-u_k(x), \quad \ F_k := f_k(x+h)-f_k(x),
\]
with $\|F_k\|_{L^q(B_k)} \le \frac{4}{k}$, $\|V_k\|_{C^{\beta}(\R^n)} \leq C_h$.

\item Let $K_0^{(k)}(y)$ be the kernel denoting the even part of $K(x, x+y)$ at $x = 0$, i.e., $K_0^{(k)}(y) := \frac12K^{(k)}(0,y)+\frac12K^{(k)}(0,-y)$.

For any $\eta\in C^\infty_c(\R^n)$ we have (for $k$ such that ${\rm supp}\,\eta\subset B_k$)
\[\int_{\R^n} \int_{\R^n} \hspace{-2mm} \big(V_k(x)-V_k(z)\big)\big(\eta(x)-\eta(z)\big) K^{(k)}_0(z-x)dz\, dx = 2\int_{\R^n}\hspace{-2mm}  F_k\eta+ E_0-E_h,\]
where
\[E_0:= \int_{\R^n} \int_{\R^n}  E_0(x, z)\, dz\,dx,\]
with 
\[
E_0(x, z) := \big(u_k(x)-u_k(z)\big)\big(\eta(x)-\eta(z)\big) \big(K^{(k)}(x,z) - K^{(k)}_0(z-x)\big),
\]
and the expression  for $E_{  h}$ is analogous, replacing $u_k(x)$, $u_k(z)$, and $K^{(k)}(x,z)$ by $u_k(x+ h)$, $u_k(z+ h)$,  and $K^{(k)}(x+ h,z+ h)$ respectively. By symmetry in the roles of $x$ and $z$,
\begin{equation}
\label{eq:absE0-cn}
\begin{split}
|E_0|& \le 2\int_{{\rm supp}\, \eta}\int_{\R^n} |E_0
(x, z)|\, dz\, dx.
\end{split}
\end{equation}

Since  $u_k$ are globally $C^\beta$ and  $\eta\in C^1(\R^n)$ we have
\[
\begin{split} \big|u_k(x)-u_k(z)\big|\,\big|\eta(x)-\eta(z)\big| & \leq  C|x-z|^\beta\min\{1, |x-z|\},
\end{split}\]
for some constant depending only on $\eta$. Hence, from \eqref{eq:absE0-cn} we have
\begin{equation}
\label{eq:E0bound-cn}
|E_0|\le C \sum_{\substack{\rho = 2^j\\ j \in \Z}}  \rho^\beta \min\{1, \rho\}\int_{{\rm supp}\, \eta}I_\rho(x) \, dx.
\end{equation}
where, as in the proof of Proposition~\ref{prop:compactness-x-div} (using the inequalities from \eqref{eq:byassumptdiv-cn}),
\[
\begin{split}
I_\rho(x) & :=  \int_{B_{2\rho}\setminus B_\rho}
\left| K^{(k)}(x, x+y)-K_0^{(k)}(y)\right|\, dy\le \frac{1}{k} \rho^{-2s}.
\end{split}
\]

Since $\eta$ has a compact support, putting it back into  \eqref{eq:E0bound-cn}, we get
\[
 |E_0|\le \frac{C}{k } \sum_{\substack{\rho = 2^j\\ j \in \Z}}  \rho^{\beta-2s} \min\{1, \rho\} \le \frac{C}{k } \left( \sum_{\substack{\rho = 2^j\\ j \in \N\cup\{0\}}}  \rho^{\beta-2s}  +\sum_{\substack{\rho = 2^{-j}\\ j \in \N}}  \rho^{1+\beta-2s} \right)  \le \frac{C}{k},
\]
since  $\beta < 2s$ and $1+\beta > 2s$.

\item We have proved $|E_0|\le \frac{C}{k}$, and the same   holds for $E_{h}$. Together with the fact that  $\|F_k\|_{L^q(B_k)} \le \frac{4}{k}$, we get 
\[\int_{\R^n} \int_{\R^n} \big(V_k(x)-V_k(z)\big)\big(\eta(x)-\eta(z)\big) K^{(k)}_0(z-x)dz\,dx \longrightarrow 0 \quad \textrm{as}\ k\to\infty.\]
By  Arzel\`a-Ascoli, the functions $V_k$ converge (up to a subsequence) locally uniformly to a function $V\in C^{\beta}(\R^n)$. 
On the other hand,    $\min\{1,|y|^2\}K^{(k)}_0(dy)$ converge weakly to some $\min\{1,|y|^2\}\bar K_0(dy)$ satisfying
\[\int_{\R^n} \int_{\R^n} \big(V(x)-V(z)\big)\big(\eta(x)-\eta(z)\big) \bar K_0(z-x)dz\, dx= 0.\]
Since $V\in C^{\beta}(\R^n)$, and $\beta > s$, it has finite energy  on compact sets, and we have that $V$ solves $\bar \L_0 V=0$ in $\R^n$ in the weak sense (where $\bar \L_0\in \GL$ is the operator with kernel $\bar K_0$), and by Liouville's theorem, Theorem~\ref{thm:Liouville}, we get that $V$ is constant.
Denoting
\[
v_k := u_k - u_k(0),
\]
then, $v_k(0)   = 0$ with $[v_k]_{C^{\beta}(\R^n)}\le 1$ and $v_k\to v$ locally uniformly, for some $v$ with $v(0) =   0$ and $[v]_{C^{\beta}(\R^n)}\le 1$, and  $V(x)=v(x+h)-v(x)$, which is constant (for every $h\in \R^n$ fixed). By Lemma~\ref{it:H10}  $v$ is an affine function, and with $[v]_{C^{\beta}(\R^n)}\le 1$ and $v(0) = 0$  it must be $v \equiv 0$, which is a contradiction with $\|v\|_{L^\infty(B_1)}\ge \eps_\circ>0$. 
\qedhere
\end{steps} 
\end{proof}

We next have the following result, which is analogous to  Proposition~\ref{prop:interior-linear-x-div}.

\begin{prop}
\label{prop:interior-linear-x-div-cn}
Let $s\in (0,1)$   $\beta\in (s,\min\{1, 2s\})$, and let $q = \frac{n}{2s-\beta}$. Let $\L$ be an operator of the form \eqref{divergence-form0}-\eqref{divergence-form1}, with kernels satisfying \eqref{divergence-ellipticity0}-\eqref{divergence-ellipticity1}.  Then, the following holds. 

For any $\delta > 0$ there exists $C_\delta$ and $\eta_\delta > 0$ such that
\[
[u]_{C^{\beta}(B_{1/2})} \le \delta [u]_{C^{\beta}(\R^n)} + C_\delta \left(\|u\|_{L^\infty(B_1)} +   \|f\|_{L^q(B_1)}\right)
\]
for any $u \in C^{\beta}_c(\R^n)$ satisfying $\L(u,x)=f$ in $B_1$ in the weak sense,  where $\L$ satisfies \eqref{reg-x-div-CN-m}-\eqref{reg-x-div-even-CN-m} with $\eta= \eta_\delta$
The constant $C_\delta$ depends only on $\delta$, $n$, $s$, $\beta$,      $\lambda$, and $\Lambda$.
\end{prop}

\begin{proof}
The proof is similar to the proof of Proposition~\ref{prop:interior-linear-x-div}.
Indeed, let us denote,  for $w\in C_c^{\beta}(\R^n)$,
\[
\tilde{\mathcal{S}} (w) := \inf\left\{\|g\|_{L^q(B_1)} :  
\begin{array}{l}
\tilde \L(w,x)=g \ \text{in the weak sense, for some $\tilde\L$ of}\\
\text{the form \eqref{divergence-form0}-\eqref{divergence-form1}-\eqref{divergence-ellipticity0}-\eqref{divergence-ellipticity1}}\\\text{and satisfying \eqref{reg-x-div}-\eqref{reg-x-div-even}, with $M>0$.}
\end{array}\right\}.
\]

By Lemma~\ref{lem-interior-blowup} with $\mu = \beta$ and
\[
\mathcal{S}(w) =  \left\{
\begin{array}{ll} \displaystyle
\tilde{\mathcal{S}} (w)  & \text{if}\quad w\in  C_c^{\beta}(\R^n).
\\
\infty & \text{otherwise,}
\end{array}
\right.
\]
Notice that the mapping $\mathcal{S}:C^{\beta}(\R^n)\to \R_{\ge 0}$ depends only on $n$, $s$, $\beta$,    $M$, $\lambda$, and~$\Lambda$.

Thus, either Lemma~\ref{lem-interior-blowup}~\ref{it:lem_int_blowup_i} holds, in which case we would have
\[
[u]_{C^{\beta}(B_{1/2})} \le \delta [u]_{C^{\beta}(\R^n)} + C_\delta \left(\|u\|_{L^\infty(B_1)} + \|f\|_{L^q(B_1)}\right),
\]
or there exists a sequence $u_k\in C^{\beta}_c(\R^n)$ and $\L_{k}$ of the previous form such that $\L_k(u_k,x)=f_k$ in the weak sense,
\begin{equation}
\label{eq:fktozero-x-div-cn}
\frac{\|f_k\|_{L^q(B_1)} }{[u_k]_{C^{\beta}(B_{1/2})}} \le \frac{2\mathcal{S}(u_k)}{[u_k]_{C^{\beta}(B_{1/2})}}\to 0,
\end{equation}
and for some $x_k \in B_{1/2}$ and $r_k \downarrow 0$,
\[
v_k(x) := \frac{u_k(x_k+r_k x)}{r_k^{\beta}[u_k]_{C^{\beta}(\R^n)}}
\]
satisfies 
\begin{equation}
\label{eq:vkcontradiction-x-div-cn}
\|v_k - v_k(0)\|_{L^\infty(B_1)} > \frac{\delta}{2}.
\end{equation}
Then, by scaling, there exists an operator of the form \eqref{divergence-form0}-\eqref{divergence-form1}-\eqref{divergence-ellipticity0}-\eqref{divergence-ellipticity1}, $\tilde \L_k$, such that 
\[
\tilde \L_k(v_k,x) = r_k^{2s-\beta}\frac{f_k(x_k+r_k x)}{ [u_k]_{C^{\beta}(\R^n)}} =:\tilde f_k(x)
\]
in the weak sense.
More precisely, if $\L_k$ has kernel $K^k(x,z)$, then $\tilde \L_k$ has kernel $\tilde K^k(x,dz)$ given by 
\[
\tilde K^k (x,dz) := r_k^{2s} K^k (x_k+r_kx,x_k+r_k \, dz).
\]
Moreover, since $\L_k$ satisfies \eqref{reg-x-div-CN-m}-\eqref{reg-x-div-even-CN-m}, $\tilde\L_k$ also satisfies them with the same constant~$\eta$. 

Also from \eqref{eq:fktozero-x-div} and since $q = \frac{n}{2s-\beta}$,
\[
\|\tilde f_k\|_{L^q (B_{1 /(2r_k)})} 
= \frac{\|f_k(x_k+r_k\,\cdot\,)\|_{L^q(B_{1 /(2r_k)})}}{r_k^{\beta-2s} [u_k]_{C^{\beta}(\R^n)}}
\leq  r_k^{2s-\beta-\frac{n}{q}}\frac{\|f_k\|_{L^q(B_1)}}{[u_k]_{C^{\beta}(\R^n)}}
\to 0, 
\] 
as $k\to \infty$. 
In all, since by definition $[v_k]_{C^\beta(\R^n)} = 1$,  we have that $v_k$ satisfies all the hypotheses of Proposition~\ref{prop:compactness-x-div-cn} for any fixed $\delta_\circ>0$, if $k$ is large enough and $\eta_\delta$ is small enough, depending on $\delta$. 
In particular, taking $\eps_\circ$ sufficiently small in Proposition~\ref{prop:compactness-x-div-cn} we get a contradiction with \eqref{eq:vkcontradiction-x-div-cn}.
\end{proof}

As a consequence, we have our desired result:
\begin{proof}[Proof of Theorem~\ref{thmD}]
It follows exactly as the proof of Theorem~\ref{thm-interior-linear-x-div} but using Proposition~\ref{prop:interior-linear-x-div-cn} instead of Proposition~\ref{prop:interior-linear-x-div}. 
\end{proof}

\section{Cordes-Nirenberg estimate in non-divergence form}
\label{sec6}

We consider now operators with $x$-dependence in non-divergence form, \eqref{x-dependence-L-nu}, but with bounded measurable coefficients that are ``close to constant''. 
Namely, instead of \eqref{x-dependence-L2}, we assume that for some small $\eta > 0$, 
\begin{equation}\label{x-dependence-L2-eta}
\sup_{\begin{subarray}{c} [\phi]_{C^\beta(\R^n)}\leq 1 \\ \phi(0)=0 \end{subarray}} 
\left|\int_{B_{2\rho}\setminus B_\rho} \phi\, d(K_x - K_{x'})\right| \leq \eta\rho^{\beta-2s},
\end{equation}
for all $x, x'\in \R^n$ and all $\rho > 0$.

Our result now reads as follows:  

\begin{thm}\label{thm-interior-linear-x-linfty}
Let $s\in (0,1)$, $\beta\in(0,2s)$, and let $\L$ be an operator of the form \eqref{x-dependence-L-nu} with $\L_x\in \GL$ for all $x\in \R^n$. 

Let $f\in L^\infty(B_1)$, and let $u\in C_{\rm loc}^{2s+}(B_1) \cap L^\infty_{2s-\eps}(\R^n)$ for some $\eps>0$ be any solution of \eqref{linear}.

Then, there exists $\eta$ small enough (depending only on $\beta$, $n$, $s$, $\lambda$, and $\Lambda$) such that if $\L$ satisfies \eqref{x-dependence-L2-eta}, then 
\[
\|u\|_{C^\beta(B_{1/2})} \le C\left(\|u\|_{L^\infty_{2s-\eps}(\R^n)} + \|f\|_{L^\infty(B_1)}\right)
\]
for some $C$ depending only on  $n$, $s$, $\beta$, $\eps$, $\lambda$, and $\Lambda$. 
\end{thm}

In order to prove Theorem \ref{thm-interior-linear-x-linfty}, we will first show the following Liouville-type theorem:  

\begin{prop}
\label{eq:Liouv_nondiv_infty}
Let $s\in (0,1)$, $\beta\in(0,2s)$ and $\beta\neq 1$, and $\delta > 0$.
 Let $\L$ be an operator of the form \eqref{x-dependence-L-nu} with $\L_x\in \GL$ for all $x\in \R^n$ satisfying \eqref{x-dependence-L2-eta} for some   $\eta\le \delta$.
 
Let $u\in C^{2s+}(\R^n)\cap L^\infty(\R^n)$ and $f\in L^\infty_{\rm loc}(\R^n)$ with $[u]_{C^{\beta}(\R^n)}\le 1$ and $\|f\|_{L^\infty(B_{{1}/{\delta}})} \le \delta$, be such that 
\[
\L(u,x) = f\quad\text{in}\quad B_{{1}/{\delta}}.
\]

Then, for every $\eps_\circ > 0$ there exists $\delta_\circ >0 $ depending only on $\eps_\circ$, $\beta$, $n$, $s$,  $\lambda$, and $\Lambda$,  such that if $\delta < \delta_\circ$,
\[
\|u-\ell\|_{C^{\lfloor \beta\rfloor}(B_1)}\le \eps_\circ,
\]
where $\ell(x)$ is the $\lfloor \beta\rfloor$-th order expansion of $u$ at $0$.
\end{prop}

\begin{proof}
We divide the proof into three steps:
\begin{steps}
\item   Assume that the statement does not hold. 
That is, there exists some $\eps_\circ > 0$ such that for any $k\in \N$, there are $u_k\in C_{\rm loc}^{2s+}(\R^n)\cap L^\infty(\R^n)$ with $[u_k]_{C^{\beta}(\R^n)}\le 1$, $f_k\in L^\infty(B_k)$ with $\|f_k\|_{L^\infty(B_k)}\le \frac{1}{k}$, and $\L^{(k)}$ as in the statement with $\L^{(k)}_x\in \GL$ for all $x\in \R^n$ satisfying \eqref{x-dependence-L2-eta} for some  $\eta\le \frac{1}{k}$, such that 
\[
\L^{(k)} (u_k, x) = f_k \quad\text{in}\quad B_k,
\]
  but 
\begin{equation}
\label{eq:contr_eps}
\|u_k - \ell_k\|_{C^{\lfloor \beta\rfloor}(B_1)}\ge \eps_\circ,
\end{equation}
where $\ell_k$ is the $\lfloor \beta\rfloor$-th expansion of $u_k$ around $0$. Let us denote 
\[
w_k := u_k-\ell_k. 
\]

We then have that $w_k(0) = 0$ and $\nabla w_k(0) = 0$ if $\beta > 1$, so that together with $[w_k]_{C^{\beta}(\R^n)}\le 1$ we get that $w_k$ converges locally uniformly in $C^{\lfloor\beta\rfloor}$ to some $w$ with $w(0) = 0$, $\nabla w(0) = 0$ if $\beta> 1$, and $[w]_{C^{\beta}(\R^n)}\le 1$. 

\item We have,
\[
\L^{(k)}(w_k, x) = f_k\quad\text{in}\quad B_k. 
\]
Let us now show that 
\begin{equation}
\label{eq:Lw0}
\L w  = 0\quad\text{in}\quad \R^n,
\end{equation}
in the viscosity sense (recall Definition~\ref{defi:viscosity}), for some $\L\in \GL$. Here, $\L$ is a weak limit (up to a subsequence) of the translation invariant operators $\L_0^{(k)}:=\L_{x=0}^{(k)}$. Namely, if we consider the finite measures 
\[
\mu_k(dy) = \min\{1, |y|^2\} K^{(k)}_0(dy), 
\]
by Prokhorov's theorem they converge weakly to some finite measure $\mu$, which defines the kernel $K(dy)$ of $\L$ to be 
\[
K(dy) := \frac{\mu(dy)}{\min\{1, |y|^2\}}. 
\]

This characterizes a limit operator (up to a subsequence) of the sequence $(\L^{(k)}_0)_{k\in \N}$. Let us now show that, indeed, \eqref{eq:Lw0} holds in the viscosity sense (recall Definition~\ref{defi:viscosity}). 

Let $\bar x\in \R^n$ be fixed, and let $\phi\in C^2(B_{3r}(\bar x))\cap L_\beta^\infty(\R^n)$  for some $r > 0$, such that $\phi(\bar x) = w(\bar x)$, $\phi < w$ in $B_{3r}(\bar x)\setminus\{\bar x\}$, and $\phi \le w$ in $\R^n$ (up to taking $\phi_\eps(x) = \phi(x) - \eps |x-\bar x|^2$ locally for some $\eps > 0$ small, we can always assume that).
Let us now define  be a sequence of test functions $\phi_k$, as follows:
\[
\phi_k := \left\{
\begin{array}{lll}
w_k & \quad \text{in}\quad \R^n\setminus B_{2r}(\bar x)\\
\phi + c_k&\quad\text{in}\quad B_r(\bar x), 
\end{array}
\right.
\]
where 
\[
c_k := \max\big\{c\in \R : \phi + c \le w_k \ \text{in}  \ B_r\big\},
\]
and defined to $B_{2r}\setminus B_r$ so that  $\phi_k\leq w_k$ and 
\begin{equation}\label{oiurgoeiru}
[\phi_k]_{C^\beta(\R^n)} \leq C(r).
\end{equation}

Since $w_k\to w$ locally uniformly and $\phi < w$ around $\bar x$, one can show $c_k \to 0$ and there exists $\bar x_k \to \bar x$ such that $\phi_k $ touches $w_k$ at $\bar x_k$ from below (see, for example, \cite[Lemma 3.5]{Silnotes} or \cite[Proof of Lemma 4.15]{FR22}). By assumption, we have 
\[
\L^{(k)}(\phi_k, \bar x_k) \ge f_k(\bar x_k) \ge - \frac{1}{k}. 
\]

In particular, 
\[
(\L^{(k)}_0\phi_k)(\bar x_k) \ge - \frac{1}{k}-\left|(\L^{(k)}_0\phi_k)(\bar x_k) -(\L^{(k)}_{\bar x_k}\phi_k)(\bar x_k) \right|. 
\]
Using \eqref{x-dependence-L2-eta} with $\eta \le \frac{1}{k}$, we claim that 
\begin{equation}\label{claim=agost}
\left|(\L^{(k)}_0\phi_k)(\bar x_k) -(\L^{(k)}_{\bar x_k}\phi_k)(\bar x_k) \right| \le \frac{C(r)}{k}.
\end{equation}
Indeed, by Lemma \ref{lem-sdhhh} (namely, \eqref{eq:Linfty_nondivform} rescaled and with $\beta$ instead of $\gamma$) we deduce that 
\[\big\|\big(\L^{(k)}_0 -\L^{(k)}_{\bar x_k}\big)\phi_k\big\|_{L^\infty(B_{r/2}(\bar x))} \le \frac{C(r)}{k}\left(\|\phi_k\|_{C^2(B_r(\bar x))} + [\phi_k]_{C^\beta(\R^n)}\right).\]
Since the right hand side is bounded by $C(r)$, and we have $\bar x_k \to \bar x$, then \eqref{claim=agost} follows.

Now, using also that $\L^{(k)}_0(\phi_k)$ is continuous around $\bar x$ uniformly in $k$ (see Lemma~\ref{lem:Lu_2}), and $\bar x_k \to \bar x$, we get 
\[
(\L^{(k)}_0\phi_k)(\bar x) \ge - \omega_k, 
\]
for some $\omega_k\downarrow 0$ as $k\to \infty$. By a similar reasoning, since $\phi_k - \phi = c_k\to 0$ in $B_r(\bar x)$, and $\phi_k \to w \ge \phi$  locally uniformly in $\R^n\setminus B_r(\bar x)$, we get for a possibly different $\omega_k$, 
\[
(\L^{(k)}_0\phi)(\bar x) \ge - \omega_k \to 0\quad\text{as}\quad k \to \infty. 
\]

Finally, from the weak convergence of the operators $\L_0^{(k)}$ to $\L$, and the fact that $\phi$ is locally $C^2$ around $\bar x$, we obtain 
\[
(\L \phi)(\bar x)\ge 0. 
\]
Since $\phi$ and $\bar x$ were arbitrary, this implies $\L w \ge 0$ in $\R^n$ in the viscosity sense. Repeating the argument from the other side (changing $w_k$ by $-w_k$, for example), we obtain that \eqref{eq:Lw0} holds in the viscosity sense, as we wanted to see. 

\item Since viscosity solutions of translation invariant operators are also distributional solutions (this can be seen, for example, regularizing with a smooth mollifier) we deduce that the previous equality is also satisfied in the distributional sense, and by Liouville's theorem for distributional solutions  (Theorem~\ref{thm:Liouville}) we obtain that $w$ is linear if $\beta>1$, and constant otherwise. 
Together with $w(0) = 0$ and $\nabla w(0) = 0$ if $\beta> 1$, we deduce that $w\equiv 0$ everywhere, which contradicts \eqref{eq:contr_eps}. \qedhere
\end{steps}
\end{proof}

We will now show the following:

\begin{prop}
\label{prop:interior-linear-x-nondiv}
Let $s\in (0,1)$, $\beta\in(0,2s)$ and $\beta\neq 1$, and let $\L$ be of the form \eqref{x-dependence-L-nu} with $\L_x\in \GL$ for all $x\in \R^n$. Then, the following holds. 

For any $\delta > 0$ there exists $C_\delta$ and $\eta_\delta > 0$ such that
\[
[u]_{C^\beta(B_{1/2})} \le \delta [u]_{C^\beta(\R^n)} + C_\delta \left(\|u\|_{L^\infty(B_1)} + \|f\|_{L^\infty(B_1)}\right)
\]
for any $u \in C^{2s+}_c(\R^n)$ with $\L(u,x)=f$ in $B_1$, where $\L$ satisfies \eqref{x-dependence-L2-eta} with $\eta = \eta_\delta$.
The constants $C_\delta$ and $\eta_\delta$ depend only on $\delta$, $n$, $s$, $\beta$, $\lambda$, and $\Lambda$.
\end{prop}

\begin{proof}
Let us denote,  for $w\in C_c^{2s+}(\R^n)$ and $\eta_\delta$ to be chosen, 
\[
\tilde{\mathcal{S}} (w) := \inf\left\{\|g\|_{L^\infty(B_1)} :  
\begin{array}{l}
\tilde \L(w,x)=g \ \text{in the viscosity sense, for some $\tilde\L$}\\
\text{of the same form as in the statement.}
\end{array}\right\}.
\]

We use Lemma~\ref{lem-interior-blowup} with $\mu = \beta$ and
\[
\mathcal{S}(w) =  \left\{
\begin{array}{ll} \displaystyle
\tilde{\mathcal{S}} (w)   & \text{if}\quad w\in  C_c^{2s+}(\R^n).
\\
\infty & \text{otherwise,}
\end{array}
\right.
\]
Notice that the mapping $\mathcal{S}:C^{2s+}(\R^n)\to \R_{\ge 0}$ depends only on $n$, $s$, $\delta$, $\beta$, $\lambda$, and $\Lambda$.

Thus, either Lemma~\ref{lem-interior-blowup}~\ref{it:lem_int_blowup_i} holds, in which case we would have
\[
[u]_{C^\beta(B_{1/2})} \le \delta [u]_{C^\beta(\R^n)} + C_\delta \left(\|u\|_{L^\infty(B_1)} + \|f\|_{L^\infty(B_1)}\right),
\]
or there exists a sequence $u_k\in C^{2s+}_c(\R^n)$ and $\L_{k}$ of the previous form such that $\L_k(u_k,x)=f_k$,
\begin{equation}
\label{eq:fktozero-x-nondiv-Linfty}
\frac{\|f_k\|_{L^\infty(B_1)}}{[u_k]_{C^\beta(B_{1/2})}} \le \frac{2\mathcal{S}(u_k)}{[u_k]_{C^\beta(B_{1/2})}}\to 0,
\end{equation}
and for some $x_k \in B_{1/2}$ and $r_k \downarrow 0$,
\[
v_k(x) := \frac{u_k(x_k+r_k x)}{r_k^\beta[u_k]_{C^\beta(\R^n)}}
\]
satisfies 
\begin{equation}
\label{eq:vkcontradiction-x-nondiv-Linfty}
\|v_k - \ell_k\|_{C^{\lfloor \beta\rfloor}(B_1)} > \frac{\delta}{2},
\end{equation}
where $\ell_k$ is the $\lfloor \beta\rfloor$-th order Taylor polynomial of $v_k$ at 0.  
Then, by scaling, there exists an operator of the form \eqref{x-dependence-L-nu}, $\tilde \L_k$, with $\tilde \L_x\in \GL$ for all $x\in \R^n$ and such that 
\[
\tilde \L_k(v_k,x) = r_k^{2s-\beta} \frac{f_k(x_k+r_k x)}{ [u_k]_{C^\beta(\R^n)}} =:\tilde f_k(x)
\]
in the viscosity sense.
More precisely, if $\L_k$ has kernel $K^k_x(dy)$, then $\tilde \L_k$ has kernel   $\tilde K^k_x(dy)$ given by 
\[
\tilde K^k_x (dy) := r_k^{ 2s} K^k_{x_k+r_k x} (r_k \, dy).
\]
Moreover, since $\L_k$ satisfies \eqref{x-dependence-L2-eta}, $\tilde\L_k$ also satisfies it with the same $\eta = \eta_\delta$. 

Also from \eqref{eq:fktozero-x-nondiv-Linfty} 
\[
\|\tilde f_k\|_{L^\infty (B_{1 /(2r_k)})} 
= r_k^{2s-\beta} \frac{\|f_k(x_k+r_k\,\cdot\,)\|_{L^\infty(B_{1 /(2r_k)})}}{[u_k]_{C^{2s}(\R^n)}}
\leq  r_k^{2s-\beta} \frac{\|f_k\|_{L^\infty(B_1)}}{[u_k]_{C^{2s}(\R^n)}}
\to 0, 
\] 
as $k\to \infty$. 
 
We are now in a situation where we can apply Liouville's theorem from Proposition~\ref{eq:Liouv_nondiv_infty} (if $\eta_\delta$ is small enough, depending on $\delta$), so that if $k$ is large enough we get a contradiction with \eqref{eq:vkcontradiction-x-nondiv-Linfty}. 
\end{proof}

As a consequence of the previous result, we have:

\begin{proof}[Proof of Theorem~\ref{thm-interior-linear-x-linfty} (and  \ref{thmA})]
 The proof is essentially the same as that of Theorem \ref{thm-interior-linear-x-div} in case $\beta<2s$, as a consequence of Proposition~\ref{prop:interior-linear-x-nondiv} (see also \ref{step3} in the proof of Theorem \ref{thm-interior-linear-x}).
Notice that we choose $\delta$ small enough (depending only on $n$, $s$, $\beta$, $\lambda$ and $\Lambda$), and this fixes the value of $\eta:=\eta_\delta$ as well. 
\end{proof}

\appendix
\section{Interpolation and incremental quotients}
\label{appA}

In this appendix we group some technical and useful lemmas on interpolation inequalities and incremental quotients that we used throughout the paper. 
They are proved, e.g., in \cite{FR23}.

We start with a result about $m$-th order incremental quotients:
\[
D_h^m f(x) = 
\left\{
\begin{array}{ll}
\frac{1}{|h|}(f(x+h)-f(x))& \quad\text{if}\quad m = 1,\\
D_h (D_h^{m-1} f(x))& \quad\text{if}\quad m\ge 2.
\end{array}
\right.
\]

\begin{lem}
\label{it:H10}
Let $f\in C^{m-1}(\R^n)$ be such that $D_h^m f$ is constant for any $h\in \R^n$ (maybe depending on $h$). Then, $f$ is a polynomial of degree $m$.
\end{lem}

And: 

\begin{lem}
\label{it:H7_gen}
Let $\alpha\in(0,1)$, $k \in \N$, $\Omega$   convex, and let $u\in  C^{k-1}(\overline\Omega)$ be such that $[u]_{C^{k-1,\alpha}(\overline\Omega)}\leq C_{\circ}$. Then,
\[
\sup_{\substack{h\in \R^n \\ \dist(x, \partial \Omega) \ge k|h|}} \frac{|D_h^k u|}{|h|^{\alpha}}\leq C C_{\circ}.
\]
For some $C$ depending only on $n$, $\alpha$, and $k$. 
\end{lem}

The following is an interpolation inequality type result. 

\begin{lem}\label{lem:interp_mult}
Let $\gamma >\beta > 0$,  let $r \in (0, \infty]$, and let $u\in C^{\gamma}(B_r)$. Then 
\[
[u]_{C^\beta(B_r)}\le C\|u\|^{1-\frac{\beta}{\gamma}}_{L^\infty(B_r)}[u]^{\frac{\beta}{\gamma}}_{C^\gamma(B_r)} +Cr^{-\beta} \|u\|_{L^\infty(B_r)},
\]
for some $C$ depending only on $n$, $\gamma$, and $\beta$. 
\end{lem}

And in the next lemma, we denote by $\delta^2_h u(x)$ the second order centered increments, 
\[
  \delta^2_h u(x) = \frac{u(x+h)+u(x-h)}{2} - u(x),
\]

\begin{lem}
\label{lem:A_imp_2}
Let $\alpha\in (0, 1]$. Then, we have the following:
\begin{enumerate}[leftmargin=*, label=(\roman*)]
\item \label{it:A:1} If $u\in C_{\rm loc}^{1,\alpha}(\R^n)$ with $[u]_{C^{1,\alpha}(\R^n)}\le 1$ then 
\[
\begin{split}
|u(x+h) - u(x) - u(h) + u(0)|& \le \min\{|h||x|^\alpha, |h|^\alpha|x|\}\\ & \le |h|^{t+\alpha(1-t)}|x|^{\alpha t + 1-t}
\end{split}
\]
and
\[
\begin{split}
\big| \delta^2_h u(x) -  \delta^2_h u(0)\big|& \le \min\{2|h|^{1+\alpha},  |h|^\alpha|x|, |h||x|^{\alpha}\}\\
& \le 2 |h|^{(1+\alpha) (1-t)+t\alpha}|x|^t
\end{split}
\]
for all $x, h\in \R^n$ and $t\in [0,1]$.

\item \label{it:A:3}  If $u\in C_{\rm loc}^{2,\alpha}(\R^n)$ with $[u]_{C^{2,\alpha}(\R^n)}\le 1$ then 
\[
\big| \delta^2_h u(x) -  \delta^2_h u(0)\big|\le \min\left\{ |h|^{2}|x|^\alpha, |h|^{1+\alpha}|x|\right\}\le |h|^{2t+(1+\alpha)(1-t)}|x|^{t\alpha + 1 - t}
\]
for all $x, h\in \R^n$ and $t\in [0, 1]$. 
\end{enumerate}
\end{lem}

\end{document}